\documentclass[psamsfonts, reqno, 11pt, letterpaper]{amsart}
\usepackage{amsfonts,amsmath, amsthm, amssymb, latexsym, epsfig}
\usepackage[all]{xy}
\usepackage{xspace}
\usepackage{comment}
\usepackage{setspace}
\usepackage{enumerate}
\usepackage[mathscr]{eucal}
\RequirePackage{pdfsync}

\usepackage{bbold}

	\advance \topmargin by -\headheight
	\advance \topmargin by -\headsep

	\evensidemargin \oddsidemargin
	\newcommand{\ftn}[3]{ #1 : #2 \rightarrow #3 }
	\newcommand{\setof}[2]{\ensuremath{\left\{ #1 \: : \: #2 \right\}}}

	\newcommand{\Z}{\ensuremath{\mathbb{Z}}}
	\newcommand{\C}{\ensuremath{\mathbb{C}}}

	\newcommand{\N}{\ensuremath{\mathbb{N}}}
	
	\newcommand{\K}{\ensuremath{\mathbb{K}}}

	\newcommand{\Prim}{\operatorname{Prim}}
	\newcommand{\primT}{\Prim^\tau}

	\theoremstyle{plain}
	\newtheorem{thm}{Theorem}[section]
	\newtheorem{lemma}[thm]{Lemma}
	\newtheorem{theor}[thm]{Theorem}
	\newtheorem{propo}[thm]{Proposition}
	\newtheorem{corol}[thm]{Corollary}

	\theoremstyle{definition}
	\newtheorem{defin}[thm]{Definition}
	
	\newtheorem{remar}[thm]{Remark}
	
	\newtheorem{examp}[thm]{Example}

	\numberwithin{equation}{section}
	\numberwithin{figure}{section}

\begin{document}
	\title[Identifying AF-algebras that are graph $C^*$-algebras]{Identifying AF-algebras that are graph $\boldsymbol{C^*}$-algebras}
	
	\author{S{\o}ren Eilers}
        \address{Department of Mathematical Sciences \\
        University of Copenhagen\\
        Universitetsparken~5 \\
        DK-2100 Copenhagen, Denmark}
        \email{eilers@math.ku.dk }
        
        \author{Takeshi Katsura}
        \address{Department of Mathematics, Faculty of Science and Technology, Keio
University, 3-14-1 Hiyoshi, Kouhoku-ku, Yokohama, Japan, 223-8522.}
  \email{katsura@math.keio.ac.jp}
	
	\author{Efren Ruiz}
        \address{Department of Mathematics\\University of Hawaii,
Hilo\\200 W. Kawili St.\\
Hilo, Hawaii\\
96720-4091 USA}
        \email{ruize@hawaii.edu}
        
         \author{Mark Tomforde}
        \address{Department of Mathematics\\University of Houston\\
Houston, Texas\\
77204- 3008, USA}
        \email{tomforde@math.uh.edu}
        \date{\today}
	
%AMS info


\begin{abstract}
We consider the problem of identifying exactly which AF-algebras are isomorphic to a graph $C^*$-algebra.  We prove that any separable, unital, Type~I $C^*$-algebra with finitely many ideals is isomorphic to a graph $C^*$-algebra.  This result allows us to prove that a unital AF-algebra is isomorphic to a graph $C^*$-algebra if and only if it is a Type~I $C^*$-algebra with finitely many ideals.  We also consider nonunital AF-algebras that have a largest ideal with the property that the quotient by this ideal is the only unital quotient of the AF-algebra.  We show that such an AF-algebra is isomorphic to a graph $C^*$-algebra if and only if its unital quotient is Type~I, which occurs if and only if its unital quotient is isomorphic to $\mathsf{M}_k$ for some natural number $k$.  All of these results provide vast supporting evidence for the conjecture that an AF-algebra is isomorphic to a graph $C^*$-algebra if and only if each unital quotient of the AF-algebra is Type~I with finitely many ideals, and bear relevance for the intrigiung question of finding $K$-theoretical criteria for when an extension
of two graph $C^*$-algebras is again a graph $C^*$-algebra.
\end{abstract}

\thanks{This research was supported by the Danish National Research Foundation through the Centre for Symmetry and Deformation (DNRF92).  The fourth author was supported by a grant from the Simons Foundation (\#210035 to Mark Tomforde).  The third and fourth authors thank the Centre de Recerca Matem\`atica for supporting each for a month-long visit during which portions of this research were completed.  Also, the third and fourth author thank the University of Houston for supporting a trip by the third author to visit the fourth author in Houston to work on this project.}

\maketitle

%%%%%%%%%%%%%%%%%%%%%%%%%%%%%%%%%%%%%%%%%%
\section{Introduction}
%%%%%%%%%%%%%%%%%%%%%%%%%%%%%%%%%%%%%%%%%%

Since the introduction of graph $C^*$-algebras in the 1990s, it has been observed that graph $C^*$-algebras contain numerous AF-algebras.  Indeed, Drinen proved that every AF-algebra is Morita equivalent to a graph $C^*$-algebra \cite{Dri}.  At the same time, it is easily seen that there are AF-algebras that are not isomorphic to any graph $C^*$-algebra.  For example, the only commutative graph $C^*$-algebras that are AF-algebras are the direct sums of complex numbers, so any commutative AF-algebra that is not isomorphic to the direct sum of copies of $\C$ (for instance, the $C^*$-algebra of continuous complex-valued functions on the Cantor set) is not isomorphic to a graph $C^*$-algebra.  This has led to the natural question of determining exactly which AF-algebras are isomorphic to graph $C^*$-algebras.

An extensive exploration of this question was undertaken by Sims together with the second and fourth named authors in \cite{kst:realization}, where they not only investigated which AF-algebras are isomorphic to graph $C^*$-algebras, but also which AF-algebras are isomorphic to Exel-Laca $C^*$-algebras, and which AF-algebras are isomorphic to ultragraph $C^*$-algebras.  A complete answer to this question for the class of graph $C^*$-algebras was not obtained in \cite{kst:realization}, although many useful partial results were deduced.  In particular, if one restricts to the class of row-finite graphs with no sinks, the question has been completely answered:  An AF-algebra is isomorphic to the $C^*$-algebra of a row-finite graph with no sinks if and only if it has no unital quotients \cite[Theorem~4.7]{kst:realization}.  In addition, an interesting necessary condition for an AF-algebra to be isomorphic to a graph $C^*$-algebra was obtained in \cite[Proposition~4.21]{kst:realization}, where it is shown that an AF graph $C^*$-algebra has the property that all of its unital quotients are Type~I $C^*$-algebras with finitely many ideals.  This naturally leads one to conjecture that the converse is true.  We state this conjecture here, and we will refer to it throughout the paper.

$ $

\noindent \textbf{Conjecture:} An AF-algebra is isomorphic to a graph $C^*$-algebra if and only if every unital quotient of the AF-algebra is a Type~I $C^*$-algebra with finitely many ideals.

$ $

As we have mentioned, \cite[Proposition~4.21]{kst:realization} establishes the ``only if" direction of the conjecture, so the open question is to determine whether the ``if" direction holds.  Also, we observe that the result in  \cite[Theorem~4.7]{kst:realization} is consistent with the conjecture, since it states that any AF-algebra with no unital quotients (which therefore vacuously satisfies the condition of the conjecture) is isomorphic to the $C^*$-algebra of a row-finite graph with no sinks.

In this paper we prove results that provide mounting evidence in support of this conjecture.  After some preliminaries in Section~\ref{prelim-sec}, we consider Type~I $C^*$-algebras in Section~\ref{unital-sec} and prove in Theorem~\ref{t:typeIgraph} that any separable, unital, Type~I $C^*$-algebra with finitely many ideals is isomorphic to a graph $C^*$-algebra.  This allows us to give a complete description of the unital AF-algebras that are isomorphic to graph $C^*$-algebras, and in Corollary~\ref{AF-conj-for-unital-cor} we prove that a unital AF-algebra is isomorphic to a graph $C^*$-algebra if and only if it is a Type~I $C^*$-algebra with finitely many ideals.  This result supports the conjecture mentioned above, since all quotients of a unital Type~I $C^*$-algebra with finitely many ideals are also unital Type~I $C^*$-algebras with finitely many ideals.  

In the remainder of the paper we consider nonunital AF-algebras that have a unital quotient.  However, this situation here is much more difficult than the unital case.  Indeed, we restrict our attention to nonunital AF-algebras that have a largest ideal (i.e., a proper ideal that contains all other ideals) with the property that the quotient by this ideal is the only unital quotient of the AF-algebra.  Studying these nonunital AF-algebras requires a subtle analysis of Bratteli diagrams, and we spend Section~\ref{Brat-sec} developing the needed technical lemmas.  In Section~\ref{nonunital-sec} we prove in Theorem~\ref{one-big-ideal-thm} that if $\mathfrak{A}$ is an AF-algebra with a largest ideal having the property that the quotient by this ideal is the only unital quotient, then $\mathfrak{A}$ is isomorphic to a graph $C^*$-algebra if and only if this unital quotient is a Type~I $C^*$-algebra, which is also equivalent to the unital quotient being isomorphic to $\mathsf{M}_k$ for some natural number $k$.  This result provides additional support for the conjecture mentioned earlier, since these AF-algebras have exactly one unital quotient, and this unital quotient is simple.  Moreover, unlike the result for unital $C^*$-algebras in Section~\ref{unital-sec}, our result in Theorem~\ref{one-big-ideal-thm} is entirely constructive, and shows exactly how to build the $C^*$-algebra from a Bratteli diagram for the AF-algebra.  

Combining our results for unital and nonunital AF-algebras, we are also able to show in Theorem~\ref{one-ideal-thm} that the conjecture from above holds for all AF-algebras with exactly one proper nonzero ideal.  Finally, we end the paper with an alternate proof of  \cite[Theorem~4.7]{kst:realization}.  The original proof in  \cite{kst:realization} shows that an AF-algebra with no unital quotients is isomorphic to a graph $C^*$-algebra in an indirect way, through the use of ultragraphs.  Our alternate proof in Theorem~\ref{t:nounitalquot} shows exactly how to construct the necessary graph from a Bratteli diagram for the AF-algebra.

The results presented here bear relevance for the intrigiung question of finding $K$-theoretical criteria for when an extension
\[
\xymatrix{
0\ar[r]&{C^*(E)}\ar[r]&{\mathfrak A}\ar[r]&{C^*(F)}\ar[r]&0}
\]
of two graph $C^*$-algebras is again a graph algebra. In the cases where one or both of $C^*(E),C^*(F)$ are purely infinite and simple, classification methods combined with range results (\cite{EKTW}, \cite{ERRfull}, \cite{EKRT}) lead to the resolution of such questions, but these methods do not apply to the case where $C^*(E)$ and $C^*(F)$ --- and then, by \cite{Brown}, also $\mathfrak A$ --- are AF. If our Conjecture holds true, this would imply that any extension of AF graph $C^*$-algebras is again an AF graph $C^*$-algebra, and our main results confirming this in key cases may be used to close the gap (cf.~\cite{EKRT}) in our present knowledge and complete the picture when both $C^*(E)$ and $C^*(F)$ are simple.

%%%%%%%%%%%%%%%%%%%%%%%%%%%%%%%%%%%%%%%%%%
\section{Background and Preliminaries} \label{prelim-sec}
%%%%%%%%%%%%%%%%%%%%%%%%%%%%%%%%%%%%%%%%%%

A \emph{graph} $E=(E^{0},E^{1},r,s)$ consists of a countable
set $E^{0}$ of vertices, a countable set $E^{1}$ of edges, and
maps $r \colon E^{1} \to E^{0}$ and $s \colon E^1 \to E^0$ identifying the
range and source of each edge.  A \emph{path} in a graph $E = (E^0, E^1, r, s)$ is a finite sequence
of edges $\alpha := e_1 \ldots e_n$ with $s(e_{i+1}) = r(e_i)$
for $1 \leq i \leq n-1$. We say that $\alpha$ has \emph{length}
$n$, and we write $|\alpha|$ for the length of $\alpha$.  We regard vertices as paths of length 0 and edges as
paths of length 1, and we then extend our notation for the
vertex set and the edge set by writing $E^n$ for the set of
paths of length $n$ for all $n \ge 0$. We write $E^*$ for the
set $\bigsqcup_{n=0}^\infty E^n$ of paths of finite length, and
extend the maps $r$ and $s$ to $E^*$ by setting $r(v) = s(v) =
v$ for $v \in E^0$, and $r(\alpha_1 \ldots \alpha_n) =
r(\alpha_n)$ and $s(\alpha_1\ldots\alpha_n) = s(\alpha_1)$.

If $\alpha$ and $\beta$ are elements of $E^*$ such that
$r(\alpha) = s(\beta)$, then $\alpha\beta$ is the path of
length $|\alpha|+|\beta|$ obtained by concatenating the two.
Given $\alpha, \beta \in E^*$, and a subset $X \subseteq E^*$, we
define
\[
\alpha X \beta := \{ \gamma \in E^* : \gamma = \alpha \gamma' \beta \text{ for some } \gamma' \in X \}.
\]
So when $v$ and $w$ are vertices, we have
\begin{align*}
vX &= \{\gamma \in X : s(\gamma) = v\},\\
Xw &= \{\gamma \in X : r(\gamma) = w\},\text{ and}\\
vXw &= \{ \gamma \in X : s(\gamma) = v \text{ and } r(\gamma) = w \}.
\end{align*}
In particular, $vE^1w$ denotes the set of edges from $v$ to $w$
and $|vE^1w|$ denotes the number of edges from $v$ to $w$.  Furthermore, if $V \subseteq E^0$, $W \subseteq E^0$, and $X \subseteq E^*$, we define
$$VXW := \{ \alpha \in X : s(\alpha) \in V \text{ and } r(\alpha) \in W \}.$$

We say a vertex $v$ is a \emph{sink} if $vE^1 = \emptyset$ and
an \emph{infinite emitter} if $vE^1$ is infinite. A \emph{singular vertex} is a vertex that is either a sink or an infinite emitter.  A graph is
called \emph{row-finite} if it has no infinite emitters.

\begin{defin}
If $E=(E^{0},E^{1},r,s)$ is a graph, then the \emph{graph
$C^*$-algebra} $C^{*}(E)$ is the universal $C^*$-algebra
generated by mutually orthogonal projections $\{p_{v} : v \in
E^{0} \}$ and partial isometries $\{ s_{e} : e \in E^{1} \}$
with mutually orthogonal ranges satisfying
\begin{enumerate}
\item $s_{e}^{*}s_{e} = p_{r(e)}$ \ \ for all $e \in E^{1}$
\item $p_{v}=\displaystyle  {\ \sum_{e \in vE^1} s_{e}
   s_{e}^*}$ \ \ for all $v\in E^{0}$ such that
   $0<|vE^1|<\infty$
\item $s_{e}s_{e}^{*} \leq p_{s(e)}$ \ \ for all $e \in
   E^{1}$.
\end{enumerate}
\end{defin}

We write $v \geq w$ to mean that there is a path $\alpha \in
E^*$ such that $s(\alpha) = v$ and $r(\alpha) = w$.
A \emph{cycle} in a graph $E$ is a path $\alpha \in E^*$ of
nonzero length with $r(\alpha) = s(\alpha)$.  A graph is called \emph{acyclic} if it has no cycles.
A graph $C^*$-algebra $C^*(E)$ is an AF-algebra
if and only if $E$ is acyclic \cite[Theorem~2.4]{KPR}.

\begin{defin}
If $E=(E^{0},E^{1},r,s)$ is a graph, a subset $H \subseteq E^0$ is called \emph{hereditary} if whenever $e \in E^1$ and $s(e) \in H$, then $r(e) \in H$.  A hereditary set $H$ is called \emph{saturated} if whenever $v \in E^0$ is a vertex that is neither a sink nor an infinite emitter, then $r (s^{-1}(v)) \subseteq H$ implies $v \in H$.

If $H$ is a saturated hereditary subset of $E^0$, we define the graph $E_H := (E_H^0, E_H^1, r_{E_H}, s_{E_H})$ as follows:  The vertex set is $E^0 := E^0 \setminus H$, the edge set is $E^1 := s^{-1}(H)$, and the range and source maps are $r_{E_H} := r|_{E_H^1}$ and $s_{E_H} := s|_{E_H^1}$, which are obtained by restricting $r$ and $s$ to $E_H^1$.
\end{defin}

%%%%%%%%%%%%%%%%%%%%%%%%%%%%%%%%%%%%%%%%%%
\section{Unital Type~I $C^{*}$-algebras and unital AF-algebras} \label{unital-sec}
%%%%%%%%%%%%%%%%%%%%%%%%%%%%%%%%%%%%%%%%%%

In this section, we prove that any unital separable Type~I $C^{*}$-algebra with finitely many ideals is isomorphic to a graph $C^{*}$-algebra.  We obtain this result in two steps:  First, we show that any unital, separable, Type~I $C^{*}$-algebra with finitely many ideals is stably isomorphic to a the $C^{*}$-algebra of an amplified graph with finitely many vertices (see Definition~\ref{d:amplified} and Proposition~\ref{p:stableiso}).  Second, we show that if $\overline{G}$ is an acyclic amplified graph with a finite number of vertices, then any full unital corner of the stabilization $C^*(\overline{G}) \otimes \K$ is isomorphic to a graph $C^*$-algebra (see Proposition~\ref{p:corners}).  To do this it will be convenient for us to apply theorems for $C^*$-algebras classified by their tempered primitive ideal space, so we begin by establishing the necessary terminology and preliminary results.

%------------------------------------------
\subsection[The tempered primitive ideal space of a $C^*$-algebra]{The tempered primitive ideal space of a $\boldsymbol{C^*}$-algebra}
%------------------------------------------

Let $X$ be a topological space and let $\mathbb{O}( X)$ be the set of open subsets of $X$ partially ordered by set inclusion.  A subset $Y \subseteq X$ is called \emph{locally closed} if $Y = U \setminus V$ where $U, V \in \mathbb{O} ( X )$ and $V \subseteq U$.  The set of all locally closed subsets of $X$ will be denoted by $\mathbb{LC}(X)$.

The partially ordered set $( \mathbb{O} ( X ) , \subseteq )$ is a lattice with meet and join given by $Y_1 \wedge Y_2 := Y_1 \cap Y_2$ and $Y_1 \vee Y_2 := Y_1 \cup Y_2$, respectively.  For a $C^{*}$-algebra $\mathfrak{A}$, we let $\mathbb{I} ( \mathfrak{A} )$ denote the set of closed ideals of $\mathfrak{A}$.  The partially ordered set $( \mathbb{I} ( \mathfrak{A} ), \subseteq )$ is also a lattice with meet and join given by $I_1 \wedge I_2 := I_1 \cap I_2$ and $I_1 \vee I_2 := \overline{I_1 + I_2}$.  If $\mathfrak{A}$ is a $C^{*}$-algebra, we let $\mathrm{Prim} ( \mathfrak{A} )$ denote the \emph{primitive ideal space} of $\mathfrak{A}$ equipped with the usual hull-kernel topology.  For any $C^*$-algebra $\mathfrak{A}$, the lattices $\mathbb{O} ( \mathrm{Prim} ( \mathfrak{A} ) )$ and $\mathbb{I} ( \mathfrak{A} )$ are isomorphic via the lattice isomorphism
\begin{equation*}
U \mapsto \bigcap_{ \mathfrak{p} \in \mathrm{Prim} ( \mathfrak{A} ) \setminus U } \mathfrak{p}
\end{equation*}
We shall frequently identify $\mathbb{O} ( \mathrm{Prim} ( \mathfrak{A} ))$ and $\mathbb{I} ( \mathfrak{A} )$ in this way.

\begin{defin}
Let $X$ be a topological space.  A \emph{$C^{*}$-algebra over $X$} is a pair $( \mathfrak{A} , \psi )$ consisting of a $C^{*}$-algebra $\mathfrak{A}$ and a continuous map $\ftn{ \psi }{ \mathrm{Prim} ( \mathfrak{A} ) }{ X }$.  
\end{defin}

If $( \mathfrak{A} , \psi )$ is a $C^{*}$-algebra over $X$, we have a map $\ftn{ \psi^{*} }{ \mathbb{O} ( X ) }{ \mathbb{O} ( \mathrm{Prim} ( \mathfrak{A} ) ) }$ defined by
\begin{equation*}
U \mapsto \setof{ \mathfrak{p} \in \mathrm{Prim} ( \mathfrak{A} ) }{ \psi ( \mathfrak{p} ) \in U }.
\end{equation*}
Using the isomorphism $\mathbb{O} ( \mathrm{Prim} ( \mathfrak{A} ) ) \cong \mathbb{I} ( \mathfrak{A} )$, we obtain a map from $\mathbb{O}(X)$ to $\mathbb{I}(\mathfrak{A} )$ given by $U \mapsto \mathfrak{A}[U]$, where
$$
\mathfrak{A}[U] := \bigcap \setof{ \mathfrak{p} \in \mathrm{Prim} ( \mathfrak{A} ) }{ \psi ( \mathfrak{p} ) \notin U }.
$$
If $Y \in \mathbb{LC} ( X )$, we may write $Y = U \setminus V$ for open sets $U, V \subseteq X$ with $V \subseteq U$, and we define $\mathfrak{A}[Y] := \mathfrak{A} [U] / \mathfrak{A}[V]$.   It follows from \cite[Lemma~2.15]{rmrn:bootstrap} that $\mathfrak{A} [Y]$ is independent of the choice of $U$ and $V$.

\begin{remar}
Any $C^*$-algebra $\mathfrak{A}$ can be viewed as a $C^*$-algebra over $\mathrm{Prim} ( \mathfrak{A} )$ by taking $\psi := \textrm{id} : \mathrm{Prim} ( \mathfrak{A} ) \to \mathrm{Prim} ( \mathfrak{A} )$.  In this case we shall simply write $\mathfrak{A}$ in place of $(\mathfrak{A}, \textrm{id})$.
\end{remar}

\begin{defin}[The Tempered Primitive Ideal Space]
Let $\mathfrak{A}$ be a $C^*$-algebra, and view $\mathfrak{A}$ as a $C^*$-algebra over $\mathrm{Prim} ( \mathfrak{A} )$.  Define $\ftn{ \tau_{ \mathfrak{A} } }{ \mathrm{Prim} (\mathfrak{A})  }{ \Z \cup \{ -\infty, \infty \} }$ by 
\[
\tau_{ \mathfrak{A} } ( x ) :=
\begin{cases}
 - \mathrm{rank} ( K_0(\mathfrak{A} [ x ]) ) &\text{if $K_{0} ( \mathfrak{A}  [ x ] )_{+} \neq K_{0} ( \mathfrak{A} [x] )$} \\
  \mathrm{rank} ( K_0(\mathfrak{A} [ x ]) ) &\text{if  $K_{0} ( \mathfrak{A}  [ x ] )_{+} = K_{0} ( \mathfrak{A} [x] )$}.
\end{cases}
\]
The \emph{tempered primitive ideal space} of $\mathfrak{A}$ is defined to be the pair $$\primT ( \mathfrak{A} ) := 
( \mathrm{Prim} ( \mathfrak{A} ), \tau_{ \mathfrak{A} } ).$$
If $\mathfrak{A}$ and $\mathfrak{B}$ are $C^*$-algebras, we say that $\primT ( \mathfrak{A} )$ and $\primT ( \mathfrak{B} ) $ are  \emph{isomorphic}, denoted by $\primT ( \mathfrak{A} )  \cong \primT ( \mathfrak{A} )$, if there exists a homeomorphism $\ftn{ \alpha }{ \mathrm{Prim} (\mathfrak{A})  }{ \mathrm{Prim} (\mathfrak{B})  }$ such that $\tau_{ \mathfrak{B} } \circ \alpha = \tau_{ \mathfrak{A} }$.
\end{defin}

\begin{defin}[Definition~6.1 of \cite{seaser:amplified}]\label{d:freeclass}
Let $\mathcal{C}$ be the class of separable, nuclear, simple, purely infinite $C^{*}$-algebras  satisfying the UCT that have free $K_0$-group and zero $K_1$-group.

Let $\mathcal{C}_{ \mathrm{free} }$ be class of all $C^{*}$-algebras $\mathfrak{A}$ satisfying all of the following four properties:
\begin{itemize}
\item[(1)] $\mathrm{Prim} ( \mathfrak{A} )$ is finite.
\item[(2)] For each $x \in \mathrm{Prim} ( \mathfrak{A} )$, the subquotient $\mathfrak{A} [ x  ]$ is either unital or  stable.
\item[(3)] For each $x \in \mathrm{Prim} ( \mathfrak{A} )$, the subquotient $\mathfrak{A} [ x  ]$ is either in $\mathcal{C}$ or stably isomorphic to $\K$. 
\item[(4)] For each $x \in \mathrm{Prim} ( \mathfrak{A} )$, if the subquotient $\mathfrak{A} [ x  ]$ is unital, then there exists an isomorphism from $K_{0} ( \mathfrak{A} [ x ] )$ onto $\bigoplus_{ n } \Z$ that takes $[ 1_{ \mathfrak{A} [ x ] } ]$ to an element of the form $( 1, \lambda ) \in \bigoplus_{ n } \Z$ for some $\lambda$. 
\end{itemize}
\end{defin}

It turns out that $C^{*}$-algebras in $\mathcal{C}_{ \mathrm{free}}$ are classified up to stable isomorphism by the tempered primitive ideal space.

\begin{theor}(Theorem 6.17 of \cite{seaser:amplified})\label{t:amplifiedfree}
If $\mathfrak{A}, \mathfrak{B} \in \mathcal{C}_{ \mathrm{free} }$, then $\mathfrak{A} \otimes \K \cong \mathfrak{B} \otimes \K$  if and only if $\primT ( \mathfrak{A} )  \cong \primT ( \mathfrak{B} )$.
\end{theor}

%----------------------
\subsection[Realizing unital Type~I  $C^{*}$-algebras as graph $C^*$-algebras]{Realizing unital Type~I  $\boldsymbol{C^{*}}$-algebras as graph $\boldsymbol{C^*}$-algebras}
%----------------------

\begin{defin}\label{d:amplified}
If $E = (E^0, E^1, r, s)$ is graph, we say that $E$ is an \emph{amplified graph} if for all $e \in E^1$ the number of edges from $s(e)$ to $r(e)$ is countably infinite.

If $G = (G^0, G^1, r, s)$ is a graph, the \emph{amplification of $G$} is defined to be the graph $\overline{G} = (\overline{G}^0, \overline{G}^1, r_{\overline{G}}, s_{\overline{G}})$ defined by $\overline{G}^{0} := G^{0}$, 
\begin{align*}
\overline{G}^{1} := \setof{ e(v,w)^{n} }{ \text{$v, w \in G^{0}$, $n \in \N$, and there exists an edge from $v$ to $w$} }, 
\end{align*}
$s_{ \overline{G} } ( e(v,w)^{n} ) := v$, and $r_{\overline{G}} ( e(v,w)^{n} ) := w$. 

Note that a graph $E$ is an amplified graph if and only if $E = \overline{G}$ for some graph $G$, and in this case the graph $G$ may always be chosen to be row-finite with the same number of vertices as $E$.  Because of this, we shall often write an amplified graph as $\overline{G}$ for a row-finite graph $G$.
\end{defin}

\begin{lemma}\label{l:type1subq}
If $\mathfrak{A}$ is a separable Type~I $C^{*}$-algebra, then every simple subquotient of $\mathfrak{A}$ is isomorphic to either $\K$ or $\mathsf{M}_{n}$ for some $n \in \N$.  If, in addition, $\mathfrak{A}$ has finitely many ideals, then $\mathfrak{A} \otimes \K \in \mathcal{C}_{ \mathrm{free} }$.
\end{lemma}

\begin{proof}
Since $\mathfrak{A}$ is a separable Type~I $C^{*}$-algebra, every simple subquotient of $\mathfrak{A}$ is a separable Type~I $C^{*}$-algebra, and hence isomorphic to either $\K$ or $\mathsf{M}_{n}$ for some $n \in \N$. Moreover, since every simple subquotient of $\mathfrak{A}$ is isomorphic to $\K$ or $\mathsf{M}_{n}$ for some $n \in \N$, we have that $(\mathfrak{A} \otimes \K)[x]$ is isomorphic to $\K$ for all $x \in \Prim ( \mathfrak{A} \otimes \K )$. Thus Property~2, Property~3, and Property~4 of Definition~\ref{d:freeclass} hold for $\mathfrak{A} \otimes \K$ (with Property~4 holding vacuously).  In addition, if  $\mathfrak{A}$ has finitely many ideals, then $\mathfrak{A} \otimes \K$ has finitely many ideals, so that $\mathrm{Prim} ( \mathfrak{A} )$ is finite, Property~1 of Definition \ref{d:freeclass} is satisfied by $\mathfrak{A} \otimes \K$, and $\mathfrak{A} \otimes \K \in \mathcal{C}_{ \mathrm{free} }$.
\end{proof}

\begin{propo}\label{p:stableiso}
If $\mathfrak{A}$ is a unital separable Type~I $C^{*}$-algebra with finitely many ideals, then there exists a finite graph $G$ such that $\mathfrak{A} \otimes \K \cong C^{*} ( \overline{G} ) \otimes \K$.
\end{propo}
       
\begin{proof}
Lemma~\ref{l:type1subq} shows that $\mathfrak{A} \otimes \K \in \mathcal{C}_{ \mathrm{free} }$ and every simple subquotient of $\mathfrak{A} \otimes \K$ is isomorphic to $\K$.  Since every simple subquotient of $\mathfrak{A} \otimes \K$ has $K_0$-group isomorphic to $\Z$ and $\mathfrak{A} \otimes \K$ has finitely many ideals, it follows that $$K_{0} ( \mathfrak{A} \otimes \K ) \cong \bigoplus_{ x \in \Prim ( \mathfrak{A} \otimes \K ) } K_{0 } ( (\mathfrak{A} \otimes \K ) [ x ] ) \cong  \bigoplus_{ x \in \Prim ( \mathfrak{A} \otimes \K ) } \Z.$$  For any $x \in \Prim ( \mathfrak{A} \otimes \K )$ we have $( \mathfrak{A} \otimes \K ) [ x ] \cong \K$ and hence $\tau_{ \mathfrak{A} \otimes \K } (x) = \{ -1 \} \subseteq \{ -1 \} \cup \N$.  Therefore \cite[Theorem~7.3]{seaser:amplified} shows there exists a finite graph $G$ such that 
\begin{equation} \label{same-prim-eq}
\primT ( \mathfrak{A} \otimes \K ) \cong \primT ( C^{*} ( \overline{G} ) \otimes \K ).
\end{equation}  
Since $\overline{G}$ is an amplified graph, every vertex of $\overline{G}$ is a singular vertex and $\overline{G}$ has no breaking vertices.  It follows from \cite[Proposition~6.10]{seaser:amplified} that $C^{*} ( \overline{G} ) \otimes \K \in \mathcal{C}_{ \mathrm{free} }$.  In addition, Lemma~\ref{l:type1subq} implies $\mathfrak{A} \otimes \K \in \mathcal{C}_{ \mathrm{free} }$.  By Theorem~\ref{t:amplifiedfree}, \eqref{same-prim-eq} implies $\mathfrak{A} \otimes \K \cong C^{*} ( \overline{G} ) \otimes \K$.
\end{proof}

Proposition~\ref{p:stableiso} shows to establish that $\mathfrak{A}$ is isomorphic to a graph $C^{*}$-algebra, it suffices to show any full unital corner of $C^{*} ( \overline{G} ) \otimes \K$ is isomorphic to a graph $C^{*}$-algebra.  We shall accomplish this by examining the range of the order unit of $C^{*} (\overline{G})$.  A more systematic study of hereditary subalgebras of graph $C^{*}$-algebras will appear in work in preparation by the third named author together with Sara Arklint and James Gabe \cite{sajger:hersubalg}. 

\begin{lemma}\label{l:isounit}
Let $\overline{G}$ be an acyclic amplified graph with a finite number of vertices, and let $S = \setof{ v \in \overline{G}^{0} }{ \text{$v$ is a source in $\overline{G}$}}$.  Let $H \subseteq \bigoplus_{ v \in \overline{G}^{0} } \Z$ be the monoid generated by 
\begin{align*}
\setof{ \delta_{v} }{ v \in \overline{G}^{0} } \cup \setof{ \delta_{v} - \sum_{ e \in T } \delta_{ r(e) } }{\text{$v \in \overline{G}^0_\textnormal{inf}$ and $T$ is a finite subset of $s^{-1}(v)$}}.
\end{align*}
If $( n_{v} )_{ v \in \overline{G}^{0} } \in H$ with $n_{v} \geq 1$ for all $v \in S$, then there exists $( m_{v} )_{ v \in \overline{G}^{0} } \in H$ with $m_{v} \geq 1$ for all $v \in \overline{G}^{0}$ and there exists an isomorphism $\ftn{ \alpha }{ \bigoplus_{ v \in \overline{G}^{0} } \Z }{ \bigoplus_{ v \in \overline{G}^{0} } \Z }$ such that $\alpha ( H ) = H$ and $\alpha \left( ( n_{v} )_{v \in \overline{G}^{0} } \right) = ( m_{v} )_{v \in \overline{G}^{0} }$.
\end{lemma}

\begin{proof}
For each $v \in \overline{G}^{0}$, set $T_{v} := \setof{ w \in S }{ w \geq v}$.  Note that $T_v$ is a finite set because $\overline{G}$ has a finite number of vertices, and $T_v$ is nonempty since $v \in T_v$.  For each $v \in \overline{G}^{0}$ and $w \in T_v$, let $k_{v,w}$ denote the smallest element of $\N \cup \{ 0 \}$ such that $n_{w} k_{v,w}  + n_{v} \geq 1$.  (Note that if $v \in S$, then $T_v = \{ v \}$ and $k_{v,v} = 0$ since $n_v \geq 1$ by hypothesis.)

Define $( m_{v} )_{ v \in \overline{G}^{0} } \in H$ by $m_v := n_v + \sum_{w \in T_v}  n_w k_{v,w}$ for $v \in \overline{G}^{0}$.  Observe that $m_v \geq 1$ for all $v \in \overline{G}^{0}$ and $m_v = n_v$ for all $v \in S$.  Also define a homomorphism  $\ftn{ \alpha }{ \bigoplus_{ v \in \overline{G}^{0} } \Z }{ \bigoplus_{ v \in \overline{G}^{0} } \Z }$ by

$$\alpha ( \delta_{v} ) = \begin{cases}  \delta_{v} + \displaystyle \sum_{ \begin{subarray}{c}w \in \overline{G}^{0}\  \text{with} \ \\ v \in T_{w} \end{subarray} } k_{w, v } \delta_{w} & \text{ if $v \in S$} \\ 
\delta_v & \text{ if $v \in  \overline{G}^{0} \setminus S$}
\end{cases}$$
One can verify that $\alpha$ is an isomorphism with inverse given by the homomorphism $\ftn{ \beta }{ \bigoplus_{ v \in \overline{G}^{0} } \Z }{ \bigoplus_{ v \in \overline{G}^{0} } \Z }$ with
\begin{align*}
\beta ( \delta_{v} ) = \begin{cases} \delta_{v} - \displaystyle \sum_{ \begin{subarray}{c}w \in \overline{G}^{0}\  \text{with} \ \\ v \in T_{w} \end{subarray} } k_{v, w } \delta_{w} & \text{ if $v \in S$} \\
\delta_v & \text{ if $v \in \overline{G}^{0} \setminus S$.}
\end{cases}
\end{align*}
In addition, 
\begin{align*}
\alpha &\left( ( n_{v} )_{v \in \overline{G}^{0} }  \right) = \alpha \left( \sum_{ v \in \overline{G}^{0} } n_{v} \delta_{v} \right) = \alpha \left( \sum_{ v \in S } n_{v} \delta_{v} \right) + \alpha \left( \sum_{ w \in \overline{G}^{0} \setminus S } n_{w} \delta_{w} \right) \\
&= \sum_{ v \in S } n_v \left( \delta_{v} +   \sum_{ \begin{subarray}{c}w \in \overline{G}^{0}\  \text{with} \ \\ v \in T_{w} \end{subarray} } k_{w, v} \delta_{w} \right) + \sum_{ w \in \overline{G}^{0} \setminus S } n_{w} \delta_{w} \\
				&=  \sum_{ v \in S } n_{v} \delta_{v}  + \sum_{ w \in \overline{G}^{0} \setminus S } \left( n_{w} + \sum_{v \in T_{w} }  n_{v} k_{w,v}  \right) \delta_{w} = \sum_{ v \in S } m_{v} \delta_{v} + \sum_{ w \in \overline{G}^{0} \setminus S } m_{w} \delta_{w}= ( m_{v} )_{ v \in \overline{G}^{0} }. 
\end{align*}

It remains to verify that $\alpha ( H ) = H$.  To do this, it suffices to show that $\alpha (H) \subseteq H$ and $\beta (H) \subseteq H$.  To establish $\alpha (H) \subseteq H$, we begin with $v \in  \overline{G}^{0}$.  By the definition of $\alpha$ and the fact that $T_w$ is finite for all $w \in \overline{G}^0$, we see that $\alpha(\delta_v) \in H$.  Next, let $v \in \overline{G}^{0}$ and let $T$ be a finite subset of $s^{-1} ( v )$.  

If $v \notin S$, then $r(e) \notin S$ for all $e \in T$, and $\alpha \left( \delta_{v} - \sum_{ e \in T } \delta_{r(e) } \right) = \delta_{v} - \sum_{ e \in T } \delta_{r(e)} \in H$.  If $v \in S$, then $r(e) \notin S$ for all $e \in T$, and 
\begin{align*}
\alpha \left( \delta_{v} - \sum_{ e \in T } \delta_{ r(e) } \right) &= \left( \delta_{v} - \sum_{ e \in T } \delta_{r(e) } \right) + \sum_{ \begin{subarray}{c}w \in \overline{G}^{0}\  \text{with} \ \\ v \in T_{w} \end{subarray} } k_{v, w } \delta_{w}\in H. 
\end{align*}
Since these elements generate $H$ and $\alpha$ is a homomorphism, we may conclude that $\alpha (H) \subseteq H$.  A nearly identical argument shows that $\beta (H) \subseteq H$.  Hence $\alpha (H) = H$.
\end{proof}

\begin{defin}
An element $a$ in a $C^{*}$-algebra $\mathfrak{A}$ is said to be \emph{full} if $a$ is not contained in a proper ideal of $\mathfrak{A}$.
\end{defin}

\begin{propo}\label{p:corners}
Let $\overline{G}$ be an acyclic amplified graph with a finite number of vertices, and let $p$ be a  full projection in $C^{*} ( \overline{G} ) \otimes \K$.  Then there exists a graph $E$ with finitely many vertices such that $p ( C^{*} ( \overline{G} ) \otimes \K ) p \cong C^{*} (E)$.
\end{propo}

\begin{proof}
It follows from \cite[Theorem~2.2]{mt:orderkthy} that $K_{0} ( C^{*} ( \overline{G} ) ) \cong \bigoplus_{ v \in \overline{G}^{0} } \Z$ via an isomorphism taking $K_{0} ( C^{*} ( \overline{G} ) )_{+}$ onto the monoid
$H \subseteq \bigoplus_{ v \in \overline{G}^{0} } \Z$ generated by 
\begin{align*}
\setof{ \delta_{v} }{ v \in \overline{G}^{0} } \cup \setof{ \delta_{v} - \sum_{ e \in T } \delta_{ r_{\overline{G} } (e) } }{\text{$v \in \overline{G}^0_\textnormal{inf}$ and $T$ is a finite subset of $s^{-1}_{\overline{G} }(v)$}}.
\end{align*}
Denote this isomorphism from $K_{0} ( C^{*} ( \overline{G} ) )$ to $\bigoplus_{ v \in \overline{G}^{0} } \Z$ by $\phi$.

Let $p$ be a  full projection in $C^{*} ( \overline{G} ) \otimes \K$.  Let $S = \setof{ v \in \overline{G}^{0} }{ \text{$v$ is a source in $\overline{G}$} }$.  Then $\phi ( [p]_0 ) = ( n_{v} )_{ v \in \overline{G} }$ such that $n_{v} \geq 0$ for all $v \in \overline{G}^{0}$.  Since $p$ is a  full projection, we must have that $n_{v} \geq 1$ for all $v \in S$.  By Lemma \ref{l:isounit}, there exists $( m_{v} )_{ v \in \overline{G}^{0} } \in H$ with $m_{v} \geq 1$ for all $v \in \overline{G}^{0}$ and there exists an isomorphism $\ftn{ \alpha }{ \bigoplus_{ v \in \overline{G}^{0} } \Z }{ \bigoplus_{ v \in \overline{G}^{0} } \Z }$ such that $\alpha ( H ) = H$ and $\alpha \left( ( n_{v} )_{v \in \overline{G}^{0} } \right) = ( m_{v} )_{v \in \overline{G}^{0} }$.

Define a directed graph as follows:  Set $$E^{0} := \overline{G}^{0} \cup \setof{ w_{k,v} }{ \text{$v \in \overline{G}^0$ and $1 \leq k \leq m_{v} - 1$} }
$$ and $$
E^{1} := \overline{G}^{1} \cup \setof{ e_{k, v } }{ \text{$v \in \overline{G}^0$ and $1 \leq k \leq m_{v} -1$} }.
$$
Define $s_{E} \vert_{\overline{G}^{1} }= s_{ \overline{G} }$ and $r_{E} \vert_{ \overline{G}^{1} } = r_{ \overline{G} }$.  Also define $s_{E} ( e_{k,v} ) = w_{k,v}$ for all $k$ and $v$, and define $r_{E} ( e_{k,v} ) = w_{k+1,v}$ when $1 \leq k \leq m_{v} - 2$ and $r_{E} ( e_{ m_{v}-1, v } ) = v$.  Using \cite[Theorem~2.2]{mt:orderkthy} and the fact that $E$ is obtained by adding finite heads at the vertices of $\overline{G}^0$, we have that $K_{0} ( C^{*} (E) )$ is isomorphic to $\bigoplus_{ v \in \overline{G}^{0} } \Z$ via an isomorphism taking $K_{0} ( C^{*} (E) )_{+}$ onto the monoid
$H \subseteq \bigoplus_{ v \in \overline{G}^{0} } \Z$ generated by $$\setof{ \delta_{v} }{ v \in \overline{G}^{0} }\cup
\setof{ \delta_{v} - \sum_{ e \in T } \delta_{ r_{\overline{G} }(e) } }{\text{$v \in \overline{G}^0_\textnormal{inf}$ and $T$ is a finite subset of $s^{-1}_{ \overline{G} } (v)$}},
$$
and furthermore, this isomorphism takes $[ 1_{ C^{*} (E) } ]_0$ to $ ( m_{v} )_{ v \in \overline{G}^{0} }$.  Denote this isomorphism from $K_{0} ( C^{*} (E) )$ to $\bigoplus_{ v \in \overline{G}^{0} } \Z$ by $\gamma$.

Let $\iota : p ( C^{*} ( \overline{G}  ) \otimes \K ) p \hookrightarrow C^{*} ( \overline{G} ) \otimes \K $ be the inclusion map.  Then 
\begin{align*}
\ftn{ K_{0} ( \iota ) }{ K_{0} \left( p ( C^{*} ( \overline{G}  ) \otimes \K ) p  \right) }{ K_{0} ( C^{*} ( \overline{G} ) ) }
\end{align*}
is an order isomorphism with $K_{0} (\iota) \left( \left[  1_{ p ( C^{*} ( \overline{G}  ) \otimes \K ) p } \right]_0 \right) = [ p ]_0$.  Hence $\gamma^{-1} \circ \alpha \circ \phi \circ K_{0} ( \iota )$ is an order isomorphism from $K_{0} ( p ( C^{*} ( \overline{G} ) \otimes \K ) p )$ to $K_{0} ( C^{*} (E) )$ with
\begin{align*}
( \gamma \circ \alpha \circ \phi \circ K_{0} ( \iota ) ) \left( \left[ 1_{ p ( C^{*} ( \overline{G} ) \otimes \K ) p } \right]_0 \right) &= ( \gamma \circ \alpha \circ \phi )( [p]_0 ) = (\gamma \circ \alpha )\left( ( n_{v} )_{v \in \overline{G}^{0} } \right) \\
&= \gamma \left( (m_{v} )_{ v \in \overline{G}^{0} } \right) = [ 1_{ C^{*} (E) } ]_0.
\end{align*}
Since $\overline{G}$ and $E$ are acyclic graphs with finitely many vertices, $p ( C^{*} (\overline{G} ) \otimes \K ) p$ and $C^{*} (E)$ are unital AF-algebras.  Hence, by Elliott's classification theorem for AF-algebras \cite{af}, we have $p ( C^{*} (\overline{G} ) \otimes \K ) p \cong C^{*} (E)$.
\end{proof}

\begin{examp}
The proof of Proposition~\ref{p:corners} actually shows how to construct the graph $E$ from the graph $G$ in the proposition's statement, provided we know the class of $[ p ]_0$ in $K_{0} ( C^{*}( \overline{G} ) )$.  For example, suppose $G$ is the graph
\begin{equation*}
\xymatrix{
\bullet \ar@{=>}[r]^\infty & \bullet.
}
\end{equation*}
and suppose $p$ is  full projection in $C^{*} ( \overline{G} ) \otimes \K$ such that $[ p ]_0$ is identified with $( 3, 2 )$ in $K_{0} ( C^{*}( \overline{G} ) ) \cong \Z \oplus \Z$.  If we define $E$ to be the graph 
\begin{align*}
\xymatrix{
\bullet \ar[r] & \bullet \ar[r] & \bullet \ar@{=>}[d]^\infty \\
		& \bullet \ar[r] & \bullet.
}
\end{align*} 
then the proof of Proposition~\ref{p:corners} shows that $p ( C^{*} ( \overline{G} ) \otimes \K ) p \cong C^{*} ( E )$.
\end{examp}

\begin{theor}\label{t:typeIgraph}
If $\mathfrak{A}$ is a separable, unital, Type~I $C^{*}$-algebra with finitely many ideals, then $\mathfrak{A}$ is isomorphic to a graph $C^{*}$-algebra.
\end{theor}

\begin{proof}
By Proposition~\ref{p:stableiso} there exists a finite graph $G$ such that $\mathfrak{A} \otimes \K \cong C^{*} ( \overline{G} ) \otimes \K$.  Let $\{ e_{ij} \}$ be a system of matrix units for $\K$.  Since $\mathfrak{A} \cong ( 1_{ \mathfrak{A} } \otimes e_{11} ) ( \mathfrak{A} \otimes \K ) ( 1_{ \mathfrak{A} } \otimes e_{11} )$ and $1_{ \mathfrak{A} } \otimes e_{11}$ is a  full projection in $\mathfrak{A} \otimes \K$, we have that $\mathfrak{A} \cong p ( C^{*} ( \overline{G} ) \otimes \K ) p$ for some full projection $p$ in $C^{*} ( \overline{G} ) \otimes \K$.  By Proposition~\ref{p:corners} $p ( C^{*} ( \overline{G} ) \otimes \K ) p \cong C^{*} (E)$ for some graph $E$.  Hence $\mathfrak{A} \cong C^{*} (E)$.
\end{proof}

We conclude this section by deducing necessary and sufficient conditions for a unital AF-algebra to be a graph $C^*$-algebra, and observe that this result provides further support for the conjecture from the introduction.

\begin{corol} \label{AF-conj-for-unital-cor}
A unital AF-algebra is isomorphic to a graph $C^{*}$-algebra if and only if it is a Type~I $C^{*}$-algebra with finitely many ideals.
\end{corol}

\begin{proof}
Suppose $\mathfrak{A}$ is a unital AF-algebra.   If $\mathfrak{A}$ is a Type~I $C^{*}$-algebra with finitely many ideals, then $\mathfrak{A}$ is isomorphic to a graph $C^*$-algebra by Theorem~\ref{t:typeIgraph}.  Conversely, if $\mathfrak{A}$ is isomorphic to a graph $C^*$-algebra, then \cite[Proposition~4.21]{kst:realization} shows that $\mathfrak{A}$ is a Type~I $C^*$-algebra with a finite number of ideals.
\end{proof}

%%%%%%%%%%%%%%%%%%%%%%%%%%%%%%%%%%%%%%%%%%
\section{Technical lemmas for Bratteli diagrams} \label{Brat-sec}
%%%%%%%%%%%%%%%%%%%%%%%%%%%%%%%%%%%%%%%%%%

In Section~\ref{nonunital-sec} we shall need a number of technical results about Bratteli diagrams for certain AF-algebras.  This section is devoted to proving these lemmas, with our primary goal being the proof of Lemma~\ref{l:bdfinal} at the end of the section.

\begin{defin}
A Bratteli diagram $(E, d_{E} )$ consists of a graph $E = ( E^{0} , E^{1}, r_{E}, s_{E} )$ and a \emph{degree function} $\ftn{ d_{E} }{ E^{0} }{ \N }$ such that 
\begin{itemize}
\item[(1)] $E$ has no sinks;

\item[(2)] $E^{0}$ is partitioned as disjoint sets $E^{0} = \bigsqcup_{ n = 1}^{ \infty } W_{n}$ with each $W_{n}$ a finite set;

\item[(3)] for each $e \in E^{1}$, there exists $n \in \N$ such that $s_{E} ( e ) \in W_{n}$ and $r_{E} ( e ) \in W_{n+1}$;

\item[(4)] for each $v \in E^{0}$, 
\begin{align*}
d_{E} (v) \geq \sum_{ e \in E^{1}v} d_{E}( s_{E} ( e ) ).
\end{align*}
\end{itemize}
We call $W_n$ the {$n$\textsuperscript{th} level} of the Bratteli diagram, and when we write $E^0 = \bigsqcup_{n=1}^\infty W_n$, we say $E^0$ is \emph{partitioned into levels} by the $W_n$.
\end{defin}

\begin{defin}
Let $(E, d_{E})$ with $E^0$ be a Bratteli diagram partitioned into levels as $E^0= \bigsqcup_{n=1}^\infty W_n$.  For any increasing subsequence $\{ n_{m} \}_{ m = 1}^{ \infty}$ of $\N$, we define a new Bratteli diagram $(F, d_{F} )$ as follows:
\begin{itemize}
\item[(1)] The set of vertices is partitioned into levels as $F^0 := \bigsqcup_{ m = 1}^{ \infty } W_{n_{m}}$;

\item[(2)] The set of edges is $F^1 := \bigcup_{m = 1}^{ \infty } W_{n_{m}} E^{*} W_{n_{m+1}}$ with the range and source map as defined on the paths of $E$; and 

\item[(3)] $d_{F} = d_{E} \vert_{ F^0}$.
\end{itemize}
We call $(F, d_{F} )$ a \emph{telescope} of $(E, d_{E})$.  

We say that two Bratteli diagrams $(E, d_{E} )$ and $( F , d_{F} )$ are \emph{equivalent} (sometimes also called \emph{telescope equivalent}) if there is a finite sequence of Bratteli diagrams  $( E_{1}, d_{E_{1} })$, \dots, $( E_{n} , d_{ E_{n} } )$ such that $( E_{1}, d_{E_{1}} ) = ( E, d_{E} )$, $( E_{n} , d_{E_{n}} ) = ( F, d_{F} )$, and for each $1 \leq i \leq n-1$, one of $( E_{i}, d_{ E_{i} } )$ and $( E_{i+1}, d_{ E_{i+1} } )$ is a telescope of the other.  Bratteli proved in \cite{ob:af} that two Bratteli diagrams give rise to isomorphic AF-algebras if and only if the diagrams are equivalent (see \cite[\S1.8 and Theorem 2.7]{ob:af}).  
\end{defin}

The following lemma is contained implicitly in the proof of \cite[Lemma~3.2]{kst:realization}.  For the convenience of the reader, we provide an explicit proof here.

\begin{lemma}\label{l:sumproj}
Let $(E, d_{E} )$ be a Bratteli diagram for the $C^*$-algebra $\mathfrak{A}$ with $E^{0}$ partitioned into levels as $E^{0} = \bigsqcup_{ n = 1}^{ \infty } W_{n}$.  Suppose $v \in W_n$ and $k \leq n$.  If $w \in E^0$ with $v \geq w$ and 
\begin{align*}
d_{E}(w)= \sum_{ \alpha \in W_{k} E^{*} w } d_{E} ( s_{E}(\alpha) ),
\end{align*}  
then 
\begin{align*}
d_{E}(v) = \sum_{ \alpha \in W_{k} E^{*} v } d_{ E } ( s_{E}(\alpha) ).
\end{align*}
In addition, if for every  $v \in E^0$ there exists $w \in E^0$ with $v \geq w$ and
\begin{align*}
d_{E}(w) = \sum_{ \alpha \in W_{1} E^{*} w } d_{ E } ( s_{E}(\alpha) ),
\end{align*}
then $\mathfrak{A}$ is a unital AF-algebra.
\end{lemma}

\begin{proof}
Note that 
\begin{align*}
d_{E}(w)  &= \sum_{ \alpha \in W_{k} E^{*} w } d_{ E } ( s_{E}(\alpha) ) = \sum_{ \beta \in W_{n } E^{*} w } \left( \sum_{ \gamma \in W_{k} E^{*} s_{E}(\beta) } d_{ E } ( s_{E}(\gamma) ) \right) \\
		&\leq \sum_{ \beta \in W_{n} E^{*} w } d_{E}(  s_{E}(\beta) ) \leq d_{E}(w)
\end{align*}
and hence we have equality throughout.  In particular, we deduce
\begin{align*}
d_{  E } ( s_{E}(\beta) ) =  \sum_{ \gamma \in W_{k} E^{*} s_{E}(\beta) } d_{ E } ( s_{E}(\gamma) ).
\end{align*}
for each $\beta \in W_{n} E^{*} w$.  Since $v \geq w$, there exists $\beta \in W_{n} E^{*} w$ such that $s_{E}( \beta ) = v$.  Thus 
\begin{align*}
d_{ E } (v) =  \sum_{ \gamma \in W_{k} E^{*} v} d_{ E } ( s_{E}(\gamma) ).
\end{align*}

For the second part of the lemma, suppose that for every  $v \in F^0$ there exists $w \in F^0$ with $v \geq w$ and
\begin{align*}
d_{E}(w) = \sum_{ \alpha \in W_{1} E^{*} w } d_{ E } ( s_{E}(\alpha) ),
\end{align*}
To show that $\mathfrak{A}$ is unital, it is enough to show that for each $v \in \bigsqcup_{n=2}^\infty W_n$ we have
\begin{align} \label{AF-unital-eq}
d_{E}(v) = \sum_{ e \in E^{1} v  } d_{ E } ( s_{E}(e) ).
\end{align} 
We shall obtain this fact by induction on $n$.  For the base case of $n=2$, we suppose $v \in W_2$.  By hypothesis there exists $w \in F^0$ such that 
$v \geq w$ and
\begin{align*}
d_{E}(w) = \sum_{ \alpha \in W_{1} E^{*} w } d_{ E } ( s_{E}(\alpha) ).
\end{align*}
It follows from the first part of this lemma, and the fact that $v \in W_2$, that
\begin{align*}
d_{E}(v) = \sum_{ \alpha \in W_{1} E^{*} v } d_{ E } ( s_{E}(\alpha) ) = \sum_{ e \in E^{1} v  } d_{ E } ( s_{E}(e) ).
\end{align*}
For the inductive step suppose $n \geq 2$ and \eqref{AF-unital-eq} holds for all vertices in $W_n$.  Let $v \in W_{n+1}$.  By hypothesis, there exists $w \in F^0$ with $v \geq w$ and 
\begin{align*}
d_{E } ( w )= \sum_{ \alpha \in W_{1} E^{*} w } d_{ E } ( s_{E}(\alpha) ).
\end{align*}
Thus
\begin{align*}
d_{E} ( v ) &= \sum_{ \beta \in W_{n} E^{*} v } \left( \sum_{ \gamma \in W_{1} E^{*} s_{E}( \beta) } d_{ E } ( s_{E}(\gamma) ) \right) \leq \sum_{ \beta \in W_{n} E^{*} v } d_{ E } (  s_{E}( \beta) ) \\
&= \sum_{ e \in E^{1} v } d_{ E } ( s_{E}(e) ) \leq d_{E} ( v ).
\end{align*}
Therefore, we have equality throughout and $d_{E}(v) = \sum_{ e \in E^{1} v } d_{ E } ( s_{E}(e) )$.  By induction \eqref{AF-unital-eq} holds for all $v \in \bigsqcup_{n=2}^\infty W_n$.
\end{proof}

\begin{lemma}\label{l:bdmatrix}
Let $(E, d_{E} )$ be a Bratteli diagram for $\mathsf{M}_{k}$ with $E^{0}$ partitioned into levels as $E^{0} = \bigsqcup_{ n = 1}^{ \infty } W_{n}$.  Then there exists $m \in \N$ such that for each $n \geq m$, $W_{n} = \{ w_{n} \}$ is a singleton set and $d_{E} (  w_{n} ) = k$.
\end{lemma}

\begin{proof}
Writing $\mathsf{M}_{k}$ as the direct limit coming from the Bratteli diagram, there exists an increasing sequence of finite-dimensional $C^{*}$-subalgebras $\mathfrak{A}_{n}$ of $\mathsf{M}_{k}$, and hence there exists $N \in \N$ such that $1_{ \mathfrak{A}_{n} } = 1_{ \mathsf{M}_{k} }$ for all $n \geq N$.  Thus, for each $n \geq N$, $k = \sum_{ v \in W_{n} } d_{E}(v)$.  Since $\mathsf{M}_{k} = \overline{ \bigcup_{ n = 1}^{ \infty } \mathfrak{A}_{n} }$, there exists $m \geq N$ such that for each $n \geq m$, $\mathrm{dim}_{\C} ( \mathsf{M}_{k} ) = \dim_{ \C } ( \mathfrak{A}_{n} )$.  Therefore, for each $n \geq m$, 
\begin{align*}
\left( \sum_{ v \in W_{n} } d_{E}(v) \right)^{2}&= k^{2} = \dim_{ \C } ( \mathsf{M}_{k} ) = \dim_{ \C } ( \mathfrak{A}_{n} ) = \sum_{ v \in W_{n} } d_{E}(v)^{2}.
\end{align*}
Since the $d_E(v)$ are non-negative integers, it follows that $W_{n} = \{ w_{n} \}$ is a singleton set and $d_{E}(w_{n}) = k$.
\end{proof}

\begin{defin}
Let $( E, d_{E} )$ be a Bratteli diagram. A saturated, hereditary subset $H$ of $E^{0}$ is a \emph{largest saturated, hereditary} subset of $E^{0}$ if whenever $Y$ is a saturated, hereditary subset  of $E^{0}$, then either $Y \subseteq H$ or $Y = E^{0}$.  
\end{defin}

\begin{defin}
Let $\mathfrak{A}$ be a $C^{*}$-algebra.  An ideal $\mathfrak{I}$ of $\mathfrak{A}$ is \emph{essential} if for every nonzero ideal $\mathfrak{K}$ of $\mathfrak{A}$, $\mathfrak{I} \cap \mathfrak{K} \neq 0$.  An ideal $\mathfrak{I}$ of $\mathfrak{A}$ is a \emph{largest ideal} of $\mathfrak{A}$ if whenever $\mathfrak{K}$ is an ideal  of $\mathfrak{A}$, then either $\mathfrak{K} \subseteq \mathfrak{I}$ or $\mathfrak{K} = \mathfrak{A}$. 
\end{defin}

\begin{remar}
Note that if $\mathfrak{I}$ is a largest ideal of $\mathfrak{A}$, then $\mathfrak{I}$ is unique, $\mathfrak{I}$ is an essential ideal of $\mathfrak{A}$, and $\mathfrak{A} / \mathfrak{I}$ is a simple $C^{*}$-algebra.
\end{remar}

%\begin{lemma} \label{largest-no-unital-quotient-lem}
%If $\mathfrak{A}$ is a $C^*$-algebra and $\mathfrak{I}$ is a largest ideal of $\mathfrak{A}$ such that $\mathfrak{A} / \mathfrak{I}$ is unital and $\mathfrak{A} / \mathfrak{I}$ is the only unital quotient of $\mathfrak{A}$, then $\mathfrak{I}$ has no unital quotients.
%\end{lemma}

%\begin{proof}
%Suppose that $\mathfrak{K}$ is an ideal of $\mathfrak{I}$ and that $\mathfrak{I} / \mathfrak{K}$ is unital.  Since $\mathfrak{I} / \mathfrak{K}$ is a unital ideal of $\mathfrak{A} / \mathfrak{K}$, we have $\mathfrak{A} / \mathfrak{K} \cong B \oplus \mathfrak{I} / \mathfrak{K}$ for some nonzero deal $B$ of $\mathfrak{A} / \mathfrak{K}$ that has trivial intersection with $\mathfrak{I} / \mathfrak{K}$.  However, since $\mathfrak{I}$ is a largest ideal of $\mathfrak{A}$, it follows that $\mathfrak{I} / \mathfrak{K}$ is a largest ideal of $\mathfrak{A} / \mathfrak{K}$, which contradicts the existence of $B$.  Hence no quotient of $\mathfrak{I}$ is unital.
%\end{proof}

\begin{defin} \label{Mk-separated-def}
For any $k \in \N$, we say a Bratteli diagram $(E,d_E)$ is \emph{$\mathsf{M}_{k}$-separated} if it satisfies the following five properties:
\begin{itemize}
\item[(1)] $E^{0} = \bigsqcup_{ n = 1}^{ \infty } W_{n}$ is partitioned into levels with $W_{n} = H_{n} \sqcup \{ y_{n} \}$.

\item[(2)] $H_{n} E^{*} y_{n+1} = \emptyset$ for all $n \in \N$.
 
\item[(3)] $d_{E} ( y_{n} ) = k$ for all $n \in \N$.

\item[(4)] $| y_{n} E^{*} y_{n+1} | = 1$ for all $n \in \N$.

\item[(5)] $y_{n} E^{*} H_{n+1} \neq \emptyset$ for all $n \in \N$.
\end{itemize}
In addition, we say that a Bratteli diagram $(E,d_E)$ is \emph{properly $\mathsf{M}_{k}$-separated} if it is $\mathsf{M}_{k}$-separated and satisfies the additional property:

\begin{itemize}
\item[(6)] For each $n \in \N$ and $v \in H_{n}$ we have
\begin{align*}
d_{E} ( v ) > \sum_{ e \in E^{1}v} d_{ E } ( s_{E}(e) ).
\end{align*}
\end{itemize}
\end{defin}

\begin{remar}
Note that if $(E,d_E)$ is an $\mathsf{M}_{k}$-separated Bratteli diagram, the set $H := \bigsqcup_{ n = 1}^{ \infty } H_{n}$ is a largest saturated hereditary subset of $E^{0}$.  In addition, if $\mathfrak{A}$ is the $C^*$-algebra associated with $(E,d_E)$, and $\mathfrak{I}$ is the ideal in $\mathfrak{A}$ associated with $H$, then $\mathfrak{I}$ is an essential ideal of $\mathfrak{A}$, and the quotient $\mathfrak{A}/\mathfrak{I}$ is an AF-algebra with Bratteli diagram $k \to k \to k \to \ldots$, so that $\mathfrak{A}/\mathfrak{I} \cong \mathsf{M}_k$.  The following lemma shows that, conversely, any AF-algebra with $\mathsf{M}_k$ as a quotient by an essential ideal has an $\mathsf{M}_k$-separated Bratteli diagram.
\end{remar}

\begin{lemma}\label{l:bdideal1}
Let $\mathfrak{A}$ be an AF-algebra with an essential ideal $\mathfrak{I}$ such that $\mathfrak{A} / \mathfrak{I} \cong \mathsf{M}_{k}$.  Then any Bratteli diagram for $\mathfrak{A}$ can be telescoped to an $\mathsf{M}_k$-separated Bratteli diagram.
\end{lemma}

\begin{proof}
Let $(E, d_{E})$ be a Bratteli diagram of $\mathfrak{A}$ with $E^{0} = \bigsqcup_{ n = 1}^{ \infty } V_{n}$ partitioned into levels, and let $S$ be the hereditary saturated subset of $E^{0}$ that corresponds to the ideal $\mathfrak{I}$.  Then the Bratteli diagram obtained by restricting to $E^{0} \setminus S = \bigsqcup_{ n = 1}^{ \infty } ( V_{n} \setminus S )$ is a Bratteli diagram for $\mathfrak{A} / \mathfrak{I}$.  Since $\mathfrak{A} / \mathfrak{I} \cong \mathsf{M}_{k}$, by Lemma~\ref{l:bdmatrix} there exists $m \in \N$ such that for each $n \geq m$, $| V_{n} \setminus S | = 1$ and $d_{ E} (x_{n} ) = k$, where $\{ x_{n} \} = V_{n} \setminus S$.  Note that for each $n \geq m$, $V_{n} = ( V_{n} \cap S ) \sqcup \{ x_{n} \}$.  To obtain the result, we shall establish two claims.

\emph{Claim 1:} There exists infinitely many $n$ such that $x_{n} E^{1} (V_{n+1} \cap S) \neq \emptyset$.  Suppose not.  Then there exists $N \geq m$ such that for all $n \geq N$, $x_{n} E^{1} (V_{n+1} \cap S ) = \emptyset$ and $x_{N} E^{1} ( V_{N+1} \cap S ) \neq \emptyset$.  Then $\bigcup_{ n = N+1}^{\infty} \{ x_{n} \}$ is a saturated, hereditary subset of $E^{0}$ disjoint from $S$, which corresponds to a nonzero ideal $\mathfrak{K}$ of $\mathfrak{A}$ such that $\mathfrak{I} \cap \mathfrak{K} = 0$, contradicting the fact that $\mathfrak{I}$ is an essential ideal of $\mathfrak{A}$. 

\emph{Claim 2:} There exists $m_{1} \geq m$ such that for all $n \geq m_{1}$ we have $( V_{n} \cap S ) E^{1} x_{n+1} = \emptyset$.  Suppose not.  Then for each $m_{1} \geq m$, there exists $n \geq m_{1}$ such that $( V_{n} \cap S ) E^{1} x_{n+1} \neq \emptyset$.  Since $S$ is a saturated, hereditary subset of $E^{0}$, this would imply that $\bigsqcup_{ n = m_{1}}^{ \infty } V_{n} \subseteq S$.  Hence, $\mathfrak{I} = \mathfrak{A}$ contradicting that fact that $\mathfrak{A} / \mathfrak{I} \cong \mathsf{M}_{k}$. 

By Claim 1 and Claim 2, there exists a subsequence $\{ k(n) \}_{ n = 1}^{ \infty }$ of $\setof{ n \in \N }{ n \geq m_{1} }$ such that $x_{k(n)} E^{*} (V_{k(n+1)} \cap S) \neq \emptyset$ and $( V_{k(n) }  \cap S ) E^{*} x_{k(n+1)} = \emptyset$.  Telescope $(E, d_{E})$ to $\bigsqcup_{ n = 1}^{ \infty } V_{k(n) }$.  Then we get a Bratteli diagram $( F, d_{F} )$ with $F^{0} = \bigsqcup_{ n = 1}^{ \infty } W_{n}$ and $W_{n} = U_{n} \sqcup \{ t_{n}  \}$, where $U_{n} := V_{k(n)} \cap S$, and $t_{n} := x_{k(n)}$.  We see that $( F, d_{F} )$ satisfies properties (1)--(5) of Definition~\ref{Mk-separated-def}.
\end{proof}

\begin{lemma}\label{l:bdnonunit}
Let $\mathfrak{A}$ be an AF-algebra with a largest ideal $\mathfrak{I}$ such that $\mathfrak{A}/\mathfrak{I} \cong \mathsf{M}_k$ and $\mathfrak{A}/\mathfrak{I}$ is the only unital quotient of $\mathfrak{A}$.  Then there exists an $\mathsf{M}_k$-separated Bratteli diagram $(F,d_F)$ for $\mathfrak{A}$ with $F^0 := \bigsqcup_{n=1}^\infty W_n$ partitioned into levels for which $W_n = U_n \cup \{ t_n \}$ satisfies Properties (1)--(5) of Definition~\ref{Mk-separated-def} and also satisfies the additional property:  
\begin{itemize} 
\item[(6')] For every $n \geq 2$ and for every $v \in U_n$ either
$$d_{F} ( v ) > \sum_{ \alpha \in W_{m} F^{*} v } d_{ F } ( s_{F}( \alpha ) ) \qquad \text{or}  \qquad d_{F}(v) = \sum_{ \alpha \in t_{m} F^{*} v} d_{ F} ( t_{m} ).$$
\end{itemize}
\end{lemma}

\begin{proof}
Note that it suffices to show that $\mathfrak{A}$ has an $\mathsf{M}_k$-separated Bratteli diagram satisfying the following condition:
\begin{itemize}
\item[(6'')] For every $m \in \N$, there exists $n \geq m$ such that for every $v \in U_{n}$ either
$$d_{F} ( v ) > \sum_{ \alpha \in W_{m} F^{*} v } d_{ F } ( s_{F}( \alpha ) ) \qquad \text{or}  \qquad d_{F}(v) = \sum_{ \alpha \in t_{m} F^{*} v} d_{ F} ( t_{m} )$$
\end{itemize}
since any such Bratteli diagram can be telescoped to an  $\mathsf{M}_k$-separated Bratteli diagram satisfying Property~(6') in the statement of the lemma.

Since any largest ideal is also an essential ideal, Lemma~\ref{l:bdideal1} implies that there is an $\mathsf{M}_k$-separated Bratteli diagram $(F, d_F)$ for $\mathfrak{A}$.  Suppose $F^0 := \bigsqcup_{n=1}^\infty W_n$ is partitioned into levels with $W_n = U_n \cup \{ t_n \}$.  We shall show that $(F, d_F)$ satisfies Property~(6'') above.  We establish this through proof by contradiction.  To this end, suppose there exists $m \in \N$ such that for each $n \geq m$, the set
\begin{align*}
Y_{n} := \setof{ x \in W_{n} }{ \text{$d_{F}(x) = \sum_{ \alpha \in W_{m} F^{*} x} d_{ F } ( s_{F}( \alpha ) )$ and $d_{F} ( x ) \neq \sum_{ \alpha \in t_{m} F^{*} x} d_{F}( t_{m} ) $}}
\end{align*}
is nonempty.  Without loss of generality, we may assume that $m = 1$.  Set 
\begin{align*}
T = \setof{ w \in F^{0} }{ \text{there exist infinitely many $n \in \N$ such that $w \geq Y_n$} }.
\end{align*}
Then $F^{0} \setminus T$ is a saturated hereditary subset of $F^{0}$.  In addition, since $\mathfrak{A}$ contains a largest ideal, and $U := \bigcup_{n=1}^\infty U_n$ is a saturated hereditary subset of $F^0$ not contained in any proper saturated hereditary subset of $F^0$, it follows that $U$ is the saturated hereditary subset corresponding to $\mathfrak{I}$ and $U$ is a largest saturated hereditary subset of $F^{0}$.  Therefore, $F^{0} \setminus T \subseteq U$ or $F^{0} \setminus T = F^{0}$.  We shall show that it must be the case that $F^{0} \setminus T \subseteq U$ by proving that $T \neq \emptyset$.

We claim that $T \cap U \neq \emptyset$.  Suppose $T \cap U = \emptyset$.  Then, $U \subseteq F^{0} \setminus T$.  Therefore, for every $v \in U$, there exists $n_{v} \in \N$ such that for each $n \geq n_{v}$, there are no paths from $v$ to $Y_{n}$.  Set $N = \max \setof{ n_{v} }{ v \in U_{1} }$.  Then for each $n \geq N$ we have $U_{1} F^{*} Y_{n} = \emptyset$.  Let $x \in Y_{N}$.  Then $d_{F}(x) = \sum_{ \alpha \in W_{1} F^{*} x} d_{ F } ( s_{F}( \alpha ) )$ and $d_{F}(x) \neq \sum_{ \alpha \in t_{1} F^{*} x} d_{F}( t_{1} )$.  Note that $x \neq t_{N}$ since $\sum_{ \alpha \in t_{1} F^{*} t_{N} } d_{F} ( t_{1} ) = d_{  F } (t_{1} ) = k = d_{F} (x)$.  Therefore, $x \in U_{N}$.  Since $0 \neq d_{F}(x) = \sum_{ \alpha \in W_{1} F^{*} x} d_{ F } ( s_{F}( \alpha ) )$, we have that $W_{1} F^{*} x \neq \emptyset$.  Since $U_{1} F^{*} Y_{N} = \emptyset$, it follows that $U_{1} F^{*} x = \emptyset$.  Hence
\begin{align*}
d_{F } ( x ) = \sum_{ \alpha \in W_{1} F^{*} x } d_{ F } ( s_{F}( \alpha ) )= \sum_{ \alpha \in t_{1} F^{*} x } d_{ F } ( t_{1}  ),
\end{align*}
which contradicts the assumption that $x \in Y_{N}$.  Therefore, $T \cap H \neq \emptyset$. 

Since $T \cap H \neq \emptyset$, we have that $T \neq \emptyset$ and $F^{0} \setminus T \neq F^{0}$.  Hence it must be the case that $F^{0} \setminus T \subseteq H$.  Since $T \cap H \neq \emptyset$, we have that $F^{0} \setminus T \neq H$.  Let $\mathfrak{K}$ be the ideal of $\mathfrak{A}$ corresponding to $F^{0} \setminus T$.  Then $\mathfrak{K} \neq \mathfrak{I}$ and $\mathfrak{A} / \mathfrak{K}$ has a Bratteli diagram obtained by restricting to the vertices of $T$.  By Lemma \ref{l:sumproj}, $\mathfrak{A} /\mathfrak{K}$ is a unital $C^{*}$-algebra, and since $\mathfrak{K} \neq \mathfrak{I}$, this contradicts the fact that $\mathfrak{A} / \mathfrak{I}$ is the only unital quotient of $\mathfrak{A}$.  Hence the lemma holds.
\end{proof}

\begin{lemma}\label{l:bdfinal}
Let $\mathfrak{A}$ be an AF-algebra with a largest ideal $\mathfrak{I}$ such that $\mathfrak{A}/\mathfrak{I} \cong \mathsf{M}_k$ and $\mathfrak{A}/\mathfrak{I}$ is the only unital quotient of $\mathfrak{A}$.  Then there exists a proper $\mathsf{M}_k$-separated Bratteli diagram for $\mathfrak{A}$.
\end{lemma}

\begin{examp}
By Lemma~\ref{l:bdnonunit} it suffices to show that an $\mathsf{M}_k$-separated Bratteli diagram that satisfies Property~(6') of Lemma~\ref{l:bdnonunit} is equivalent to a proper $\mathsf{M}_k$-separated Bratteli diagram for $\mathfrak{A}$.  To help the reader follow the proof of Lemma~\ref{l:bdfinal}, we give an example to illustrate how the telescoping constructions in the proof are performed. 

Below are four Bratteli diagrams: $(F,d_F)$, $(A,d_A)$, $(B,d_B)$, and $(E,d_E)$. The Bratteli diagram $(F,d_F)$ is an $\mathsf{M}_1$-separated Bratteli diagram that satisfies Property~(6') of Lemma~\ref{l:bdnonunit}.  In addition $(E,d_E)$ is a proper $\mathsf{M}_1$-separated Bratteli diagram.  Telescoping the Bratteli diagrams $(F,d_F)$ and $(B,d_{B})$ at the odd levels, we obtain the Bratteli diagram $(A,d_{A})$.  Telescoping the Bratteli diagram $(B,d_{B})$ at the even levels gives the Bratteli diagram $(E, d_E)$.  Thus $(F,d_F)$ is equivalent to $(E,d_E)$.

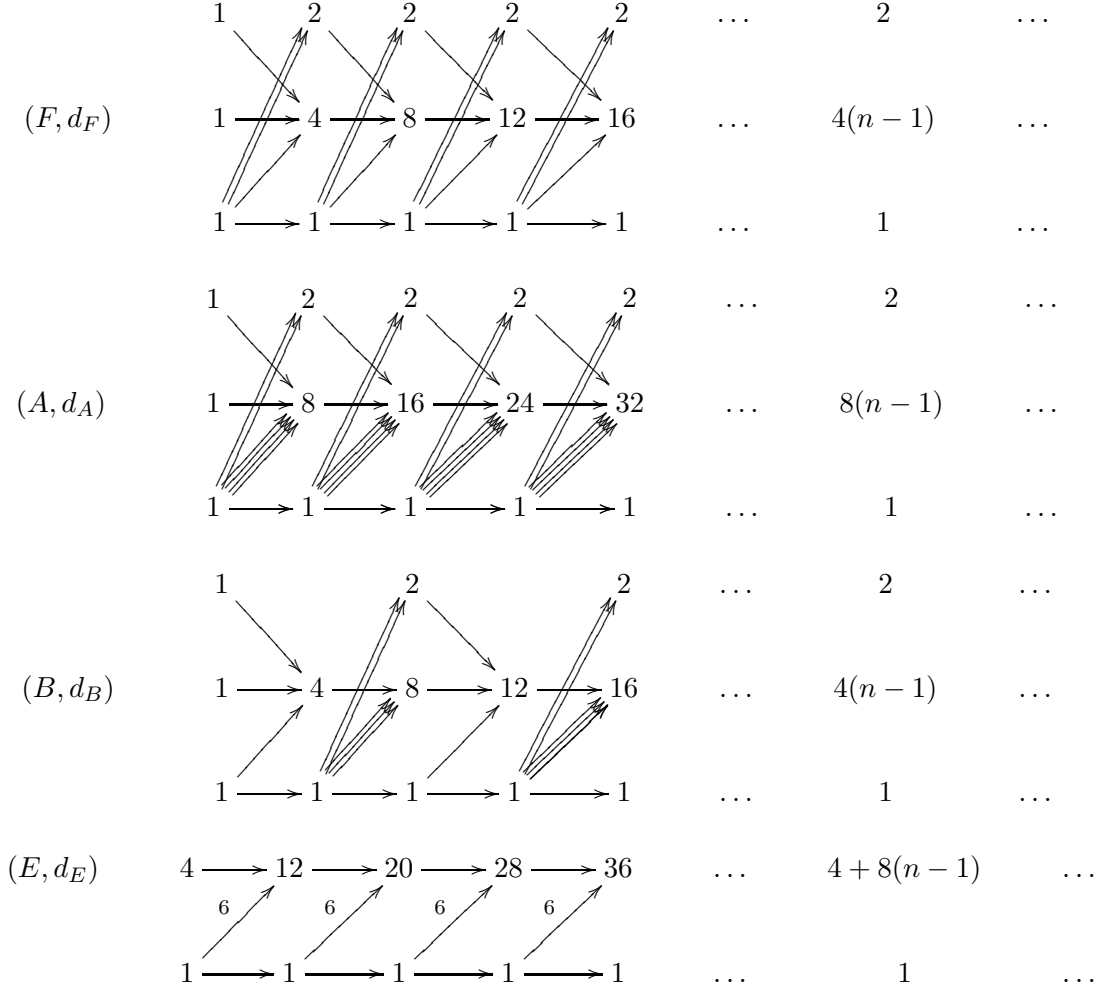
\begin{center}
$ $ \xymatrix{
& 1 \ar[rd] 			 & 2 	\ar[rd]  & 2 \ar[rd] & 2 \ar[rd] & 2   &\dots & 2 & \dots \\
(F,d_F) \ \ & 1 \ar[r] 			 & 4  \ar[r]   & 8  \ar[r]  & 12  \ar[r]  & 16 & \dots& 4(n-1) & \dots \\
& 1 \ar@<.5ex>[ruu]  \ar[ruu] \ar[ur] \ar[r] & 1 \ar@<.5ex>[ruu]  \ar[ruu]  \ar[ur] \ar[r] & 1 \ar@<.5ex>[ruu]  \ar[ruu]  \ar[ur] \ar[r]  & 1 \ar@<.5ex>[ruu]  \ar[ruu]  \ar[ur] \ar[r]  & 1 &  \dots  & 1 & \dots
}
\end{center}

$ $

\begin{center}
$ $ \xymatrix{
& 1 \ar[rd] 			 & 2 	\ar[rd]  & 2 \ar[rd] & 2 \ar[rd] & 2 & \dots & 2 & \dots\\
(A,d_{A}) \ \ & 1 \ar[r] 			 & 8  \ar[r]   & 16  \ar[r]  & 24 \ar[r] & 32 & \dots & 8(n-1) & \dots \\
& 1   \ar[r]  \ar@<.5ex>[ruu]  \ar[ruu]  \ar@<.8ex>[ur] \ar@<.4ex>[ur] \ar[ur] \ar@<-.4ex>[ur] \ar[r]  & 1 \ar[r]  \ar@<.5ex>[ruu]  \ar[ruu]  \ar@<.8ex>[ur] \ar@<.4ex>[ur] \ar[ur] \ar@<-.4ex>[ur] \ar[r]  &1 \ar[r]  \ar@<.5ex>[ruu]  \ar[ruu]  \ar@<.8ex>[ur] \ar@<.4ex>[ur] \ar[ur] \ar@<-.4ex>[ur] \ar[r] & 1 \ar[r]  \ar@<.5ex>[ruu]  \ar[ruu]  \ar@<.8ex>[ur] \ar@<.4ex>[ur] \ar[ur] \ar@<-.4ex>[ur] \ar[r] & 1 &  \dots & 1 & \dots
}
\end{center}

$ $

\begin{center}
$ $ \xymatrix{
& 1 \ar[rd] 			 &   & 2 \ar[rd] &  & 2  &\dots & 2 & \dots \\
(B,d_{B}) \ \ & 1 \ar[r] 			 & 4  \ar[r]   & 8  \ar[r]  & 12  \ar[r]  & 16 & \dots & 4(n-1) & \dots \\
& 1   \ar[ur] \ar[r] & 1 \ar@<.5ex>[ruu]  \ar[ruu]  \ar@<.8ex>[ur] \ar@<.4ex>[ur] \ar[ur] \ar[r] & 1  \ar[ur] \ar[r]  & 1 \ar@<.5ex>[ruu]  \ar[ruu]  \ar[ur] \ar[r]  \ar@<.8ex>[ur] \ar@<.4ex>[ur] \ar[ur] & 1 &  \dots & 1 & \dots
}
\end{center}

$ $

\begin{center}
$ $ \xymatrix{
\quad (E,d_E) &  4 \ar[r] 			 & 12  \ar[r]   & 20 \ar[r]  & 28 \ar[r] & 36  & \dots  & 4+8(n-1) & \dots  \\
\quad & 1  \ar[ur]^-{6} \ar[r]  & 1 \ar[ur]^-{6}  \ar[r]    &1 \ar[ur]^-{6} \ar[r]   & 1\ar[ru]^-{6} \ar[r]   & 1 & \dots & 1 & \dots 
}
\end{center}

\end{examp}

\noindent \emph{Proof of Lemma~\ref{l:bdfinal}.}
By Lemma~\ref{l:bdnonunit} $\mathfrak{A}$ has an $\mathsf{M}_k$-separated Bratteli diagram $(F, d_F)$ that satisfies Property~(6') of Lemma~\ref{l:bdnonunit}.  We shall prove that $(F, d_F)$ is equivalent to a proper $\mathsf{M}_k$-separated Bratteli diagram for $\mathfrak{A}$.

For each $n \geq 2$, set
\begin{align*}
A_{n} := \setof{v \in U_{n} }{ d_{F} (v) = \sum_{ e \in t_{n-1} F^{1} v } d_{F}( t_{n-1} ) }.
\end{align*}
For each $v \in A_{n}$ and for each $w \in U_{n+1}$, let 
\begin{align*}
p(v,w) := | \setof{ ef \in F^2}{ s_F(e) = t_{n-1}, r_{F}(e) = s_{F} ( f ) = v, \ \text{and} \ r_{F} ( f ) = w } |
\end{align*}
denote the number of paths in $F$ from $t_{n-1}$ to $w$ that go through $v$.

Define a Bratteli diagram $(B,d_{B})$ as follows:
\begin{align*}
B^{0} &= F^{0} \setminus \left( \bigcup_{ n = 1}^{ \infty } A_{2n} \right)\\
B^{1} &= \left( F^{1} \setminus \left( \bigcup_{ n = 1}^{ \infty } \setof{ e \in F^{1} }{ \text{$r_{F}(e) \in A_{2n}$ or $s_{F} ( e ) \in A_{2n} $} } \right) \right) \\
	&\qquad \qquad \sqcup \left( \bigcup_{ n = 1}^{\infty } \setof{ e_{i} (v,w, n) }{ v \in A_{2n}, w \in U_{2n+1} , \text{ and }1 \leq i \leq p(v,w) } \right)
\end{align*}
with range and source maps defined by
\begin{align*}
r_{B} ( e ) = 
\begin{cases}
r_{F} (e),	&\text{if $e \in F^{1}$} \\
w,		 &\text{if $e = e_{i} ( v, w, n )$} 
\end{cases}
\quad \text{and} \quad
s_{B} ( e ) = 
\begin{cases}
s_{F} ( e ),	&\text{if $e \in F^{1}$} \\
t_{2n},	&\text{if $e = e_{i} ( v, w, n )$}
\end{cases}
\end{align*}
and the degree function defined by $d_{B} = d_{F}\vert_{B^0}$.  

Note that $B^{0} = \bigsqcup_{ n = 1}^{ \infty } ( V_{n} \sqcup \{ t_{n} \} )$ with $V_{2n-1} = U_{2n-1}$ and $V_{2n} = U_{2n} \setminus A_{2n}$.  To show that $(B, d_{B} )$ is a Bratteli diagram, we must show that 
\begin{align*}
d_{B} ( v ) \geq \sum_{ e \in B^{1} v } d_{B} ( s_{B} ( e ) ).
\end{align*}
for all $v \in B^0$.  To do this it suffices to show that for each $w \in V_{n} \sqcup \{ t_{n} \}$ we have
\begin{align*}
\sum_{ e \in B^{1} w } d_{B} ( s_{B} ( e ) ) = \sum_{ e \in F^{1} w } d_{F} ( s_{F} ( w ) ).
\end{align*}
To this end, let $n \in \N$ and first suppose $w \in V_{2n}$.  Then $w \in U_{2n} \setminus A_{2n}$.  By the construction of $B$, we have that 
$B^{1}w = F^{1} w$. Hence 
\begin{align*}
\sum_{ e \in B^{1} w } d_{B} ( s_{B} ( e ) ) = \sum_{ e \in F^{1} w} d_{F} ( s_{F} ( e ) ). 
\end{align*}
Next, suppose $w \in V_{2n-1}$.  If $w \in A_{2n-1}$ with $n \geq 2$ (the case when $n =1$ is clear), then $B^{1} w = F^{1} w$ and
\begin{align*}
\sum_{ e \in B^{1} w } d_{B} ( s_{B} ( e ) ) = \sum_{ e \in F^{1} w} d_{F} ( s_{F} ( e ) ). 
\end{align*}
Next, suppose $w \in U_{2n-1} \setminus A_{2n-1}$.  Note that $d_{F} ( t_{2n-3} ) = d_{F} ( t_{2n-2} )$.  Thus 
\begin{align*}
&\sum_{ e \in B^{1} w } d_{B} ( s_{B} ( e ) )  \\
&\quad = \sum_{ e \in V_{2n-2} B^{1} w} d_{B} ( s_{B} ( e ) ) + \sum_{ e\in t_{2n-2} B^{1} w} d_{B} ( s_{B} ( e ) ) \\
&\quad = \sum_{ e \in V_{2n-2} F^{1} w} d_{F} ( s_{F} ( e ) ) + \sum_{ e \in t_{ 2n-2 } F^{1} w } d_{F} ( s_{F} ( t_{2n-2} ) ) +  \sum_{ v \in A_{2n-2} } s( v, w ) d_{F} ( t_{2n-2} ) \\
&\quad = \sum_{ e \in V_{2n-2} F^{1} w } d_{F} ( s_{F} (e) ) +  \sum_{ e \in t_{ 2n-2 } F^{1} w } d_{F} ( s_{F} ( t_{2n-2} ) ) + \sum_{ v \in A_{2n-2 } } | t_{2n-3} F^{1} v | | v F^{1} w | d_{F} ( t_{2n-2} ) \\
&\quad =  \sum_{ e \in V_{2n-2} F^{1} w } d_{F} ( s_{F} (e) ) +  \sum_{e \in t_{ 2n-2 } F^{1} w } d_{F} ( s_{F} ( t_{2n-2} ) ) + \sum_{ v \in A_{2n-2} } \sum_{ e \in t_{2n-3} F^{1} v } | v F^{1} w | d_{F} ( t_{2n-3 } )   \\
&\quad = \sum_{ e \in V_{2n-2} F^{1} w } d_{F} ( s_{F} (e) ) +  \sum_{ e \in t_{ 2n-2 } F^{1} w } d_{F} ( s_{F} ( t_{2n-2} ) ) + \sum_{ v \in A_{2n-2} } | v F^{1} w | d_{F} (v) \\
& \quad = \sum_{ e \in V_{2n-2} F^{1} w } d_{E} ( s_{E} (e) ) +  \sum_{e \in t_{ 2n-2 } E^{1} w } d_{E} ( s_{E} ( t_{2n-2} ) )  + \sum_{ e \in A_{2n-2} F^{1} w } d_{F} ( s_{F} ( e ) ) \\
&\quad = \sum_{ e \in F^{1} w } d_{F} ( s_{F} ( e ) ).
\end{align*}
Since
\begin{align*}
\sum_{ e \in B^{1} t_{n} } d_{B} ( s_{B} (e) ) = d_{F} ( t_{n-1} ) = \sum_{ e \in F^{1}t_{n} } d_{F} ( s_{F} (e) ),
\end{align*}
it follows that for each $w \in B^{0}$ we have
\begin{align*}
\sum_{ e \in B^{1} w } d_{B} ( s_{B} ( e ) ) = \sum_{ e \in F^{1} w } d_{F} ( s_{F} ( e ) ) \leq d_{F} ( w ) = d_{B} ( w ).
\end{align*}
Thus $( B , d_{B} )$ is a Bratteli diagram.

Set $Y_{n} = V_{n} \sqcup \{ t_{n} \}$.  By the construction of $(B, d_{B} )$, for each $w \in A_{2n+1}$ and for each $v \in A_{2n-1}$ we have
\begin{align*}
| v B^{*} w | = | v F^{*} w |.
\end{align*}  
Let $(A, d_A)$ be the the Bratteli diagram obtained by telescoping $(B, d_{B} )$ at the levels $\bigsqcup_{ n = 1}^{ \infty } Y_{2n-1}$.  Since 
$Y_{2n-1} = W_{2n-1}$ for each $n \in \N$, we see that the Bratteli diagram obtained by telescoping $( F, d_{F} )$ at the levels $\bigsqcup_{ n = 1}^{ \infty } W_{2n-1}$ is also equal to $(A,d_A)$.  Thus $( F, d_{F} )$ and $(B, d_{B} )$ are equivalent Bratteli diagrams.

We now show that for each $n \geq 2$ and $w \in Y_{2n}$ we have
\begin{align*}
\sum_{ e \in Y_{2n-2 } B^{*} w } d_{B} ( s_{B} ( e ) ) = \sum_{ e \in W_{2n-2} F^{*} w } d_{F} ( s_{F} ( e ) ). 
\end{align*}
Let $w \in V_{2n}$.  First, suppose $w = t_{2n}$.  Then 
\begin{align*}
\sum_{ e \in Y_{2n-2} B^{*} w } d_{B} ( s_{B} ( e ) ) = d_{F} ( t_{2n} ) = \sum_{ e \in W_{2n-2} F^{*} w } d_{F} ( s_{F} ( e ) ).
\end{align*}

Next, suppose $w \in V_{2n}$.  Then $w \in U_{2n} \setminus A_{2n}$ and $B^{1} w = F^{1} w$.  Therefore
\begin{align*}
\sum_{ e\in W_{2n-2} F^{*} w } d_{F} ( s_{F} ( e ) ) &= \sum_{ e \in F^{1} w } \sum_{ f \in F^{1} s_{F} (e) } d_{F} ( s_{F} ( e ) ) = \sum_{ e \in B^{1} w } \sum_{ f \in F^{1} s_{B} ( e ) } d_{F} ( s_{F} ( e ) ) \\
&= \sum_{ e \in B^{1} w } \sum_{ f \in B^{1} s_{B} ( e ) } d_{B} ( s_{B} ( e ) ) = \sum_{ e \in Y_{2n-2} B^{* } w }  d_{B} ( s_{B} ( e ) ).
\end{align*}
Hence, for each $n \geq 2$ and $w \in Y_{2n}$ we have
\begin{align*}
\sum_{ e \in Y_{2n-2 } B^{*} w } d_{B} ( s_{B} ( e ) ) = \sum_{ e \in W_{2n-2} F^{*} w } d_{F} ( s_{F} ( e ) ). 
\end{align*}
In particular, for each $n \geq 2$ and $w \in V_{2n} = U_{2n} \setminus A_{2n}$ we have
\begin{align*}
\sum_{ e \in Y_{2n-2 } B^{*} w } d_{B} ( s_{B} ( e ) ) = \sum_{ e \in W_{2n-2} F^{*} w } d_{F} ( s_{F} ( e ) ) < d_{F} ( w ) = d_{B} (w).	 
\end{align*}

Let $(E, d_{E} )$ be the Bratteli diagram obtained by telescoping $(B, d_{B} )$ to $\bigsqcup_{ n = 1}^{ \infty } Y_{2n}$.  Set $H_{n} := U_{2n}$ and $y_{n} := t_{2n}$.  Then $(E, d_{E} )$ is a proper $\mathsf{M}_k$-separated Bratteli diagram equivalent to $(F, d_F)$.
\hfil \qed

%%%%%%%%%%%%%%%%%%%%%%%%%%%%%%%%%%%%%%%%%%
\section{Nonunital AF-algebras with a unique unital quotient} \label{nonunital-sec}
%%%%%%%%%%%%%%%%%%%%%%%%%%%%%%%%%%%%%%%%%%

We begin by describing a way to construct a graph from a proper $\mathsf{M}_k$-separated Bratteli diagram.  

\begin{defin}\label{d:bdfinal}
Let $(E, d_{E})$ be a proper $\mathsf{M}_k$-separated Bratteli diagram such that $E^{0} = \bigsqcup_{ n = 1}^{ \infty } V_{n}$ is partitioned into levels with $V_{n} = H_{n} \sqcup \{ y_{n} \}$.  Then $H := \bigsqcup_{ n = 1}^{ \infty } H_{n}$ is a saturated hereditary subset of $E^{0}$, and we construct a graph $G = (G^0, G^1, r_G, s_G)$ from $(E, d_E)$ as follows:
For each $n \in \N$ and $v \in H_{n}$, set
\begin{align*}
\delta(v) := d_{E} ( v) - \sum_{ e \in r^{-1} (v) } d_{ E} ( s_{E} (e) ) - 1 \quad \text{and} \quad m(v) := | y_{n-1} E^{*} v |.
\end{align*}
Let
\begin{align*}
G^{0} &:= E_{H}^{0} \sqcup \setof{ z_{i} }{ i = 1, \dots, k } \sqcup \setof{ x_{i}^{v} }{ v \in H , 1 \leq i \leq \delta(v) } \\
G^{1} &:= E_{H}^{1} \sqcup \setof{ e_{i} }{ i = 1, \dots, k-1} \sqcup \setof{ f_{i}^{v} }{ v \in H , 1 \leq i \leq m(v) } \\
& \hspace{2.6in} \sqcup \setof{ g_{i}^{v} }{ v \in H , 1 \leq i \leq \delta(v) }
\end{align*}
be the vertex and edge sets of $G$, respectively, with
\begin{align*}
s_{G} (e) = 
\begin{cases}
s_{E} (e),	&\text{if $e \in E_{H}^{1}$} \\
z_{i},		&\text{if $e = e_{i}$} \\
z_{k},	&\text{if $e = f_{i}^{v}$} \\
x_{i}^{v},	&\text{if $e = g_{i}^{v}$}
\end{cases}
\qquad \text{and} \qquad 
r_{G} (e) = 
\begin{cases}
r_{E} (e),	&\text{if $e \in E_{H}^{1}$} \\
z_{k}, 	&\text{if $e = e_{i}$} \\
v,		&\text{if $e = f_{i}^{v}$ or $e =g_{i}^{v}$}
\end{cases}
\end{align*}
as the range and source functions. 
\end{defin}

\begin{remar}
In Definition~\ref{d:bdfinal} the fact that $(E, d_E)$ is a proper $\mathsf{M}_k$-separated Bratteli diagram is needed to assure us that $\delta(v) \geq 0$ for all $v \in H$.
\end{remar}

\begin{remar}
In the graph $G$ of Definition~\ref{d:bdfinal} the vertex $z_k$ is an infinite emitter, and all other vertices of $G$ are regular vertices.
\end{remar}

\begin{examp}
We give an example to illustrate the construction of the graph in Definition~\ref{d:bdfinal}. Consider the following proper $\mathsf{M}_3$-separated Bratteli diagram.  In the top row the values at the first three levels are $4$, $24$, $43$, and then for level $n \geq 4$ the value is $20n+4$.

$$
 
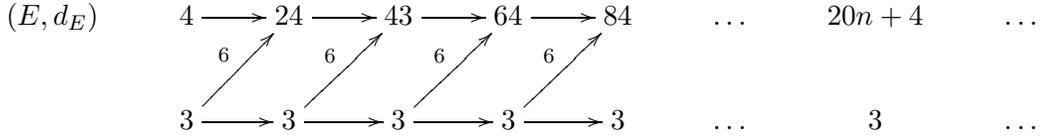
\xymatrix{
( E, d_{E} ) & 4 \ar[r] 			 & 24  \ar[r]   & 43 \ar[r]  & 64 \ar[r] & 84  & \dots & 20n+4 & \dots  \\
& 3  \ar[ur]^-{6} \ar[r]  & 3 \ar[ur]^-{6}  \ar[r]    &3 \ar[ur]^-{6} \ar[r]   & 3\ar[ru]^-{6} \ar[r]   & 3 &  \dots & 3 & \dots
}
$$
We see that $H$ consists of the vertices labeled $4, 24, 43, 64, 84, \ldots$ and that each $H_n$ consists of a single vertex, which we shall denote $v_n$.  Then $H = \{ v_1, v_2, v_3, \ldots \}$, and $\delta (v_1) = 3$, $\delta(v_2) = 1$, $\delta(v_3) = 0$, $\delta (v_4) = 2$, $\ldots$.  The graph $G = ( G^{0} , G^{1}, r_{G} , s_{G} )$ constructed from $(E, d_{E})$, as described in Definition~\ref{d:bdfinal}, is given by the following:

$$
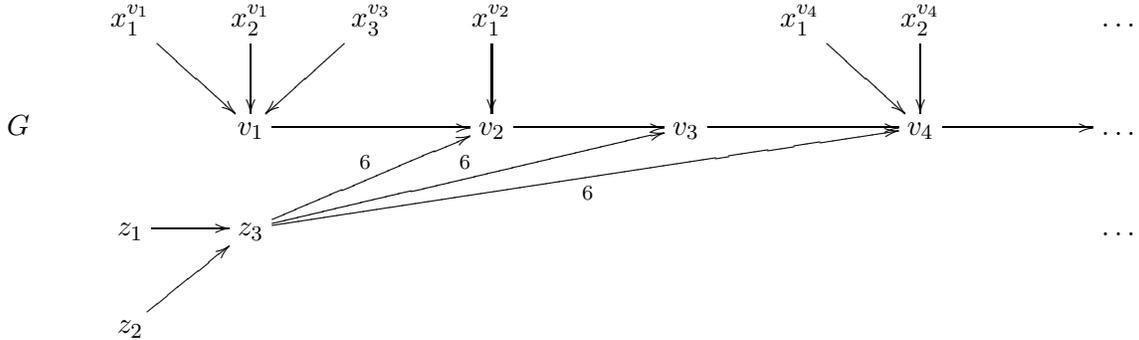
\xymatrix{
& x_1^{v_1}  \ar[rd]	& x_2^{v_1} \ar[d]			&  x_3^{v_3} \ar[ld]  	 &    		x_1^{v_2} \ar[d]	 & 		 & 	    &  
x_1^{v_4} \ar[dr] & x_2^{v_4} \ar[d] & &  \dots \\ 
G & & v_1 \ar[rr] 	&		 		& v_2  	\ar[rr]			& 			           & v_3        \ar[rr]    &  	& v_4  \ar[rr] &  & \dots  \\
& z_1 \ar[r] & z_3 \ar[rru]^{6}	\ar[rrrru]^{6}  \ar[rrrrrru]_{6} 	& 		&  		 &    & 	 &     &     & &  \dots \\
& z_2 \ar[ru] & 	& 		&  		 &    & 	 &     &     & &  
}
$$
Note that the vertex $z_3$ is an infinite emitter.
\end{examp}

\begin{lemma}\label{l:paths}
Let $(E, d_{E})$ be a a proper $\mathsf{M}_k$-separated Bratteli diagram such that $E^{0} = \bigsqcup_{ n = 1}^{ \infty } V_{n}$ is partitioned into levels with $V_{n} = H_{n} \sqcup \{ y_{n} \}$.  If we let $H := \bigsqcup_{ n = 1}^{ \infty } H_{n}$ and let $G = ( G^{0}, G^{1}, s_{G}, r_{G} )$ the graph constructed from $(E,d_E)$ as described in Definition~\ref{d:bdfinal}, then
\begin{align*}
d_{E}(v) = | \setof{ \alpha \in G^{*} }{ r_{G} ( \alpha ) = v } | \qquad \text{ for all $v \in H$.}
\end{align*}
\end{lemma}

\begin{proof}
We prove that for each $n \in \N$ we have $d_{E} (v) = | \setof{ \alpha \in G^{*} }{ r_{G} ( \alpha ) = v } |$ for all $v \in H_n$.  We accomplish this by induction on $n$.  

For the base case, we suppose $v \in H_{1}$.  Since $r_{G}^{-1} ( v ) = \setof{ g_{i}^{v} }{ 1 \leq i \leq \delta(v) }$ and $s_{G} ( g_{i}^{v} )$ is a source in $G$ for each $i$, we have
\begin{align*}
| \setof{ \alpha \in G^{*} }{ r_{G} ( \alpha ) = v } | &= 1 + | r_{G}^{-1} ( v ) | = 1 + \delta (v) = d_{E}(v) - \sum_{ e \in r_{E}^{-1} (v) } d_{ E } ( s(e) ) = d_{E} ( v )
\end{align*}
and the base case holds.  For the inductive step, assume that for a particular value of $n \in \N$ we have $d_{E} (v) = | \setof{ \alpha \in G^{*} }{ r_{G} ( \alpha ) = v } |$ for all $v \in H_n$.   Choose $v \in H_{n+1}$.  Then 
\begin{align*}
| \setof{ \alpha \in G^{*} }{ r_{G} ( \alpha ) = v } | &= 1 + | \setof{ g_{i}^{v} }{ 1 \leq i \leq \delta(v) } | \\
								&\qquad + \sum_{ e \in r_{E}^{-1}(v) \cap s_{E}^{-1} (H) } | \setof{ \alpha \in G^{*} }{ r_{G} ( \alpha ) = s_{G}(e) } | \\
								& \qquad \qquad + k| \setof{ f_{i}^{v} }{ 1 \leq i \leq m(v) } | \\ 
							&= 1 + \delta(v) + \sum_{ e \in r_{E}^{-1} (v) \cap s_{E}^{-1}(H) } d_{ E} ( s_{G} (e) ) + km(v) \\
							&= 1+ \delta(v) + \sum_{ e \in r_{E}^{-1}(v) } d_{ E} ( s_{E}(e) ) \\
							&= d_{E}(v)
\end{align*}
and our lemma holds for all vertices in $H_{n+1}$.  It follows from induction that the lemma holds.
\end{proof}

\begin{propo}\label{t:largestideal}
Let $(E, d_{E})$ be a proper $\mathsf{M}_k$-separated Bratteli diagram, and let $G$ be the graph constructed from $(E,d_E)$ as described in Definition~\ref{d:bdfinal}.  If $\mathfrak{A}$ is the AF-algebra associated with $(E,d_{E})$, then $\mathfrak{A} \cong C^{*} ( G )$.
\end{propo}

\begin{proof}
Let $\setof{ S_{e}, P_{v} }{ e \in G^{1}, v \in G^{0} }$ be a universal Cuntz-Krieger $G$-family in $C^{*} (G)$.  Using the notation of Definition~\ref{d:bdfinal} set $V_{n} := H_{n} \sqcup \{ y_{n} \}$ and define
$$
Q_{n} := P_{z_{k}} - \sum_{ 1 \leq i \leq n } \sum_{ v \in H_{i} } \sum_{ 1 \leq j \leq m(v) } S_{ f_{j}^{v} } S_{ f_{j}^{v} }^{*} $$
and
$$
\mathfrak{A}_{n} := C^{*} ( \setof{ S_{ \alpha } }{ \alpha \in E^{*}, r_{G} ( \alpha ) \in H_{n} }  \cup \setof{ S_{e_{i}} Q_{n} }{ i = 1, \dots, k-1 } ).
$$
We will prove the following:
\begin{itemize}
\item[(1)] $\mathfrak{A}_{n} \subseteq \mathfrak{A}_{n+1}$ and there exists an isomorphism 
\begin{align*}
\ftn{ \phi_{n} }{ \mathfrak{A}_{n} }{ \left( \bigoplus_{ v \in H_{n} } \mathsf{M}_{ d_{E}(v) } \right) \oplus \mathsf{M}_{k} }
\end{align*} 
such that the induced homomorphism 
\begin{align*}
\ftn{ \psi_{n,n+1} }{ \left( \bigoplus_{ v \in H_{n} } \mathsf{M}_{ d_{E}(v) } \right) \oplus \mathsf{M}_{k} } { \left( \bigoplus_{ v \in H_{n+1} } \mathsf{M}_{ d_{E}(v) } \right) \oplus \mathsf{M}_{k} }
\end{align*}
that makes the diagram 
\begin{align*}
\xymatrix{
\mathfrak{A}_{n} \ar@{^{(}->}[rr] \ar[d]^{ \phi_{n} } & &\mathfrak{A}_{n+1} \ar[d]^{ \phi_{n+1} } \\
\big( \bigoplus_{ v \in H_{n} } \mathsf{M}_{ d_E(v) } \big) \oplus \mathsf{M}_{k} \ar[rr]_{ \psi_{n,n+1} } & & \big( \bigoplus_{ v \in H_{n+1} } \mathsf{M}_{ d_{E}(v) } \big) \oplus \mathsf{M}_{k}
}
\end{align*}
commutative has multiplicity matrix $( | w E^{1} v | )_{ w \in V_{n}, v \in V_{n+1} }$. 

\item[(2)] $\setof{ S_{e}, P_{v} }{ e \in G^{1}, v \in G^{0} } \subseteq \bigcup_{ n = 1}^{ \infty } \mathfrak{A}_{n}$. 
\end{itemize}

Note that (1) implies that $\mathfrak{A} \cong \overline{ \bigcup_{ n = 1}^{ \infty } \mathfrak{A}_{n} }$ and (2) implies that $\overline{ \bigcup_{ n = 1}^{ \infty } \mathfrak{A}_{n} } = C^{*} (G)$, from which it follows that $\mathfrak{A} \cong C^{*} ( G )$.  Thus establishing (1) and (2) will prove the theorem.

We first prove (1).  Let $\alpha \in G^{*}$ such that $r_{G} ( \alpha ) = v \in H_{n}$.  Note that 
\begin{align*}
S_{\alpha} &= S_{\alpha} P_{ v } = S_{ \alpha } \sum_{ e \in s_{G}^{-1} (v) } S_{ e } S_{e}^{*} = \sum_{ e \in s_{G}^{-1} (v) } S_{ \alpha } S_{e} S_{e}^{*}.
\end{align*}
Since $v \in H_{n}$, it follows that $r_{G} ( s_{G}^{-1} (v) ) \in H_{n+1}$.  Thus, $S_{ \alpha } S_{e} S_{e}^{*} \in \mathfrak{A}_{n+1}$ for all $e \in s_{G}^{-1} (v)$.  Note that 
\begin{align*}
S_{e_{i}} Q_{n} &= S_{ e_{i} } \left( Q_{n+1} - \sum_{v \in H_{n+1} } \sum_{ 1 \leq j \leq m(v) }S_{ f_{j}^{v} } S_{f_{j}^{v}}^{*} \right) = S_{e_{i}} Q_{n+1} - \sum_{ v \in H_{n+1} } \sum_{ 1 \leq j \leq m(v) } S_{e_{i}} S_{ f_{j}^{v} } S_{f_{j}^{v}}^{*}.
\end{align*}
Since $S_{e_{i}} Q_{n+1} \in \mathfrak{A}_{n+1}$ and $S_{e_{i}} S_{ f_{j}^{v} } S_{f_{j}^{v}}^{*} \in \mathfrak{A}_{n+1}$, we have that $S_{e_{i} } Q_{n} \in \mathfrak{A}_{n+1}$.  Thus $\mathfrak{A}_{n} \subseteq \mathfrak{A}_{n+1}$.

For each $v \in H_{n}$, set 
\begin{align*}
\mathfrak{B}_{n,v} := C^{*} ( \setof{ S_{\alpha} }{ r_{G}(\alpha) = v } ) \qquad \text{and} \qquad \mathfrak{C}_{n} := C^{*} ( \setof{ S_{e_{i}} Q_{n} }{ 1 \leq i \leq k-1 } ).
\end{align*} 
Define $\ftn{ \alpha_{n} }{ \left( \bigoplus_{ v \in H_{n} } \mathfrak{B}_{n,v} \right) \oplus \mathfrak{C}_{n} }{ \mathfrak{A}_{n} }$ by 
\begin{align*}
\alpha_{n} \left( ( (x_{v})_{ v \in H_{n} } , y) \right) := \sum_{ v \in H_{n}} x_{v} + y.
\end{align*}
One can verify that $\alpha_{n}$ is an isomorphism.  In addition, for each $v \in H_{n}$, we have $d_{E}(v) = | \setof{ \alpha \in G^{*} }{ r_{G} ( \alpha ) = v } |$, and hence $\mathfrak{B}_{n,v} \cong \mathsf{M}_{d_{E}(v)}$.  Also, $\mathfrak{C}_{n} \cong \mathsf{M}_{k}$.  Let $\ftn{ \phi_{n} }{ \mathfrak{A}_{n} }{ \left( \bigoplus_{ v \in H_{n} } \mathsf{M}_{ d_{E}(v) } \right) \oplus \mathsf{M}_{k} }$ be the composition
\begin{align*}
\xymatrix{
\mathfrak{A}_{n} \ar[r]^-{ \alpha_{n}^{-1} }  & \left( \bigoplus_{ v \in H_{n} } \mathfrak{B}_{n,v} \right) \oplus \mathfrak{C}_{n} \cong \left( \bigoplus_{ v \in H_{n} } \mathsf{M}_{ d_{E}(v) } \right) \oplus \mathsf{M}_{k}.
}
\end{align*}
Let $ \iota_{n,n+1} :  \mathfrak{A}_{n} \hookrightarrow \mathfrak{A}_{n+1}$ be the inclusion map.  Define $\ftn{ \psi_{n,n+1} }{ \left( \bigoplus_{ v \in H_{n} } \mathsf{M}_{ d_{E}(v) } \right) \oplus \mathsf{M}_{k} }{ \left( \bigoplus_{ v \in H_{n+1} } \mathsf{M}_{ d_{E}(v) } \right) \oplus \mathsf{M}_{k} }$ to be the composition $\phi_{n+1} \circ \iota_{n,n+1} \circ \phi_{n}^{-1}$.

If $w \in H_{n}$, then 
\begin{align*}
P_{w} &= \sum_{ e \in s_{G}^{-1} ( w ) } S_{e} S_{e}^{*} = \sum_{ v \in H_{n+1} } \sum_{ e \in w E^{*} v} S_{e} S_{e}^{*} \in \sum_{ v \in H_{n+1} } \mathfrak{B}_{n+1,v}.
\end{align*}
Note that $$Q_{n} = Q_{n+1} - \sum_{ v \in H_{n+1} } \sum_{ 1 \leq j \leq m(v) } S_{ f_{j}^{v} } S_{f_{j}^{v}}^{*} \in \mathfrak{C}_{n+1} + \sum_{ v \in H_{n+1} } \mathfrak{B}_{n+1,v}.$$  Therefore, the multiplicity matrix $\phi_{n,n+1}$ is given by $( | w E^{*} v | )_{ w \in V_{n}, V_{n+1}}$.  This establishes (1), which implies $\mathfrak{A} \cong \overline{ \bigcup_{ n = 1}^{ \infty } \mathfrak{A}_{n} }$.

We now prove (2).  Note that for each $v \in H_{n}$, we have $S_{ g_{i}^{v} }, S_{e} \in \mathfrak{A}_{n}$ for all $e \in r_{G}^{-1} (v)$ and for all $1 \leq i \leq \delta(v)$.  Since $S_{ g_{i}^{v} } S_{g_{i}^{v}}^{*} = P_{ x_{i}^{v} }$ and $S_{ e }^{*} S_{e} = P_{ v }$ for each $e \in r_{G}^{-1} (v)$, we have that $S_{ g_{i}^{v} } , S_{e}, P_{v}, P_{ x_{i}^{v} } \in \bigcup_{ n = 1}^{ \infty } \mathfrak{A}_{n}$ for all $n \in \N$, $v \in H_{n}$, and $e \in r_{G}^{-1} (v)$.  

All that remains is to show that $S_{e_i} \in \bigcup_{ n = 1}^{ \infty } \mathfrak{A}_{n}$ for $i =1, \ldots, k-1$ and $P_{z_k} \in \bigcup_{ n = 1}^{ \infty } \mathfrak{A}_{n}$ for $k=1, \ldots k$ are in $\bigcup_{ n = 1}^{ \infty } \mathfrak{A}_{n}$.  We shall actually show that all these elements are in $\mathfrak{A}_{1}$.  To do this, we see that for each $v \in H_1$ we have $m(v) = 0$ so that $Q_1 = P_{z_k}$.  Thus $P_{z_k} \in \mathfrak{A}_{1}$.  In addition, for all $1 \leq i \leq k-1$ we have $S_{e_i} = S_{e_i}P_{z_k} = S_{e_i} Q_1 \in \mathfrak{A}_{1}$.  Moreover, it follows that for all $1 \leq i \leq k-1$ we have $P_{z_i} = S_{e_i} S_{e_i}^* \in \mathfrak{A}_{1}$.  

The previous two paragraphs show that $\setof{ S_{e} , g_{v} }{ e \in G^{1}, v \in G^{0} } \subseteq \bigcup_{ n = 1}^{ \infty } \mathfrak{A}_{n}$, which establishes (2).  Since each $\mathfrak{A}_{n}$ is a $C^*$-subalgebra of $C^*(G)$, and the elements of the set $\setof{ S_{e} , P_{v} }{ e \in G^{1}, v \in G^{0} }$ generate $C^*(G)$, it follows that $C^*(G) \cong \overline{ \bigcup_{ n = 1}^{ \infty } \mathfrak{A}_{n} }$.
\end{proof}

The following theorem shows that the conjecture from the introduction holds whenever $\mathfrak{A}$ is an AF-algebra with a largest ideal $\mathfrak{I}$ such that $\mathfrak{A} / \mathfrak{I}$ is the only unital quotient of $\mathfrak{A}$.

\begin{theor} \label{one-big-ideal-thm}
Let $\mathfrak{A}$ be a nonunital AF-algebra with a largest ideal $\mathfrak{I}$. If $\mathfrak{A}/\mathfrak{I}$ is the only unital quotient of $\mathfrak{A}$, then the following are equivalent:
\begin{itemize}
\item[(1)] $\mathfrak{A}$ is isomorphic to a graph $C^*$-algebra.
\item[(2)] $\mathfrak{A} / \mathfrak{I}$ is a Type~I $C^*$-algebra with finitely many ideals.
\item[(3)] $\mathfrak{A} / \mathfrak{I} \cong \mathsf{M}_k$ for some $k \in \N$.
\end{itemize}
\end{theor}

\begin{proof}
If (1) holds, then it follows from \cite[Proposition~4.21]{kst:realization} that the unital quotient $\mathfrak{A} / \mathfrak{I}$ is a Type~I $C^*$-algebra with finitely many ideals.  Hence (1) implies (2).

If (2) holds, then since $\mathfrak{I}$ is a largest ideal of $\mathfrak{A}$, the quotient $\mathfrak{A} / \mathfrak{I}$ is simple.  In addition, since $\mathfrak{A}$ is AF, and hence Type~I, it follows that $\mathfrak{A} / \mathfrak{I}$ is Type~I.  Since any unital, simple, Type~I $C^*$-algebra is isomorphic to $\mathsf{M}_k$ for some $k \in \N$, (3) holds.  Thus (2) implies (3)

If (3) holds, then Lemma~\ref{l:bdfinal} implies there exists a proper $\mathsf{M}_k$-separated Bratteli diagram for $\mathfrak{A}$.  If $G$ is the graph constructed from $(E,d_E)$ as described in Definition~\ref{d:bdfinal}, then  Proposition~\ref{t:largestideal} implies that $A \cong C^*(G)$.  Hence (3) implies (1).
\end{proof}

\begin{remar}
Recall that \cite[Theorem~4.7]{kst:realization} implies that an AF-algebra is isomorphic to the $C^*$-algebra of a row-finite graph with no sinks if and only if the AF-algebra has no unital quotients.  If the conditions of Theorem~\ref{one-big-ideal-thm} hold, then since $\mathfrak{A} / \mathfrak{I}$ is a unital quotient of $\mathfrak{A}$, we know that $\mathfrak{A}$ is not isomorphic to the $C^*$-algebra of a row-finite graph with no sinks.  The construction of Definition~\ref{d:bdfinal} shows that $\mathfrak{A}$ is, however, isomorphic to the $C^*$-algebra of a graph with no sinks and exactly one infinite emitter.
\end{remar}

It is easy to see that the conjecture from the introduction holds for simple AF-algebras: If a simple AF-algebra is unital, Corollary~\ref{AF-conj-for-unital-cor} shows it is isomorphic to a graph $C^*$-algebra if and only if it is a Type~I $C^*$-algebra.  If a simple AF-algebra is nonunital, \cite[Theorem~4.7]{kst:realization} shows it is always isomorphic to a graph $C^*$-algebra.

Combining our result for unital AF-algebras in Corollary~\ref{AF-conj-for-unital-cor} and for nonunital AF-algebras in Theorem~\ref{one-big-ideal-thm} allows us to show that the conjecture from the introduction also holds for the class of AF-algebras with exactly one ideal.

\begin{theor} \label{one-ideal-thm}
Let $\mathfrak{A}$ be an AF-algebra with exactly one proper nonzero ideal.  Then $\mathfrak{A}$ is isomorphic to a graph $C^*$-algebra if and only if every unital quotient of $\mathfrak{A}$ is a Type~I $C^*$-algebra.
\end{theor}

\begin{proof}
Necessity follows from \cite[Proposition~4.21]{kst:realization}.  To see sufficiency, let $\mathfrak{I}$ be the unique proper nonzero ideal of $\mathfrak{A}$, and consider three cases.

\noindent \textsc{Case~I:} $\mathfrak{A}$ is unital.  Then by hypothesis $\mathfrak{A}$ is Type~I with finitely many ideals, so by Corollary~\ref{AF-conj-for-unital-cor} we have $\mathfrak{A}$ is isomorphic to a graph $C^*$-algebra.

\noindent \textsc{Case~II:} $\mathfrak{A}$ is nonunital, and $\mathfrak{A} / \mathfrak{I}$ is unital.  Then $\mathfrak{A} / \mathfrak{I}$ is the only unital quotient of $\mathfrak{A}$, and this quotient is simple and Type~I by hypothesis.  Hence Theorem~\ref{one-big-ideal-thm} implies $\mathfrak{A}$ is isomorphic to a graph $C^*$-algebra.

\noindent \textsc{Case~III:} $\mathfrak{A}$ is nonunital and $\mathfrak{A} / \mathfrak{I}$ is nonunital.   Then $\mathfrak{A}$ has no unital quotients, and \cite[Theorem~4.7]{kst:realization} implies $\mathfrak{A}$ is isomorphic to a graph $C^*$-algebra.
\end{proof}

It was proven in \cite[Theorem~4.7]{kst:realization} that an AF-algebra with no unital quotients is isomorphic to a graph $C^*$-algebra.  However, the argument uses ultragraphs, and it is difficult to determine the required graph from the proof.  Here we provide an alternate proof that shows explicitly how to construct a graph from a Bratteli diagram whose $C^*$-algebra is isomorphic to the AF-algebra.

\begin{theor}[cf.~Theorem~4.7 of \cite{kst:realization}]\label{t:nounitalquot}
Let $A$ be an AF-algebra that has no nonunital quotients.  Then $A$ has a Bratteli diagram $(E, d_{E} )$ such that $d_E(v) \geq 2$ and $d_{E} ( v ) > \sum_{ e \in r_{E}^{-1} ( v ) } d_{E} ( s_{E} ( e ) )$ for all $v \in E^{0}$.  For any such Bratteli diagram, construct a graph $G$ from $(E,d_E)$ as follows:  Set $\delta(v) = d_{E} ( v ) -   \sum_{ e \in r_{E}^{-1} ( v ) } d_{E} ( s_{E} ( e ) ) - 1$ for each $v \in E^0$. Define the vertex set and edge set of $G$ as
$$
G^{0} = E^{0} \sqcup \setof{ x_{i}^{v} }{ v \in E^{0} , 1 \leq i \leq \delta(v) }$$
and
$$G^{1} = E^{1} \sqcup \setof{ e_{i}^{v} }{ v \in E^{0} , 1 \leq i \leq \delta(v) },$$
respectively.  Also define the range and source maps of $G$ as
$$
r_{G} \vert_{ E^{1} } = r_{E^{1}} \quad \text{and} \quad r_{G} ( e_{i}^{v} ) = v $$
and
$$s_{G} \vert_{ E^{1} } = r_{ E^{1 } } \quad \text{and} \quad s_{G} ( e_{i}^{v} ) = x_{i}^{v},$$
respectively.  Then $\mathfrak{A} \cong C^{*} ( G )$.
\end{theor}

\begin{proof}
It follows from \cite[Lemma~3.5]{kst:realization} that $A$ has a Bratteli diagram $(E, d_{E} )$ such that $d_E(v) \geq 2$ and $d_{E} ( v ) > \sum_{ e \in r_{E}^{-1} ( v ) } d_{E} ( s_{E} ( e ) )$ for all $v \in E^{0}$.

Let $E^{0}$ be partitioned into levels as $E^{0} = \bigsqcup_{ n = 1}^{ \infty } V_{n}$, let $\{ S_{e}, P_{v} \}_{ e \in G^{1}, v \in G^{0} }$ be a universal Cuntz-Krieger $G$-family for $C^{*} ( G )$, and let $\mathfrak{A}$ be the AF-algebra associated to $(E,d_{E})$.
Set 
\begin{align*}
\mathfrak{A}_{n} &= C^{*} ( \setof{ S_{ \alpha } }{ r_{G} ( \alpha ) \in V_{n} } ).
\end{align*}
Using an argument similar to the one in the proof of Proposition~\ref{t:largestideal}, we obtain that $\mathfrak{A}_{n}$ is a $C^{*}$-subalgebra of $C^{*} ( G )$ for each $n \in \N$, and the following statements hold:
\begin{itemize}
\item[(1)] $\mathfrak{A}_{n} \subseteq \mathfrak{A}_{n+1}$ for all $n \in \N$;

\item[(2)] $\mathfrak{A}_{n} \cong \bigoplus_{ v \in V_{n} } \mathsf{M}_{ d_{E} (v) }$ for all $n \in \N$;

\item[(3)] $\overline{ \bigoplus_{ n = 1}^{ \infty } \mathfrak{A}_{n} } = C^{*} ( G )$; and 

\item[(4)] the homomorphism $\ftn{ \phi_{n,n+1} }{  \bigoplus_{ v \in V_{n} } \mathsf{M}_{ d_{E} (v) } }{  \bigoplus_{ v \in V_{n+1} } \mathsf{M}_{ d_{E} (v) } }$ given by 
\begin{align*}
 \bigoplus_{ v \in V_{n} } \mathsf{M}_{ d_{E} (v) } \cong \mathfrak{A}_{n} \subseteq \mathfrak{A}_{n+1} \cong  \bigoplus_{ v \in V_{n+1} } \mathsf{M}_{ d_{E} (v) }
\end{align*}
has multiplicity matrix $( | v E^{*} w | )_{ v \in V_{n} , w \in V_{n+1} }$.
\end{itemize}
Hence, 
$$
C^{*} ( G ) = \overline{ \bigcup_{ n = 1}^{ \infty } \mathfrak{A}_{n} } \cong \varinjlim \left(  \bigoplus_{ v \in V_{n} } \mathsf{M}_{ d_{E} (v) } , \phi_{n,n+1} \right) \cong \mathfrak{A}.
$$
\end{proof}

\end{document}